\documentclass{amsart}
\usepackage[utf8]{inputenc}
\usepackage{amsmath}
\usepackage{amssymb}
\usepackage{caption}
\usepackage{amsthm}
\usepackage[hidelinks]{hyperref}
\usepackage{cleveref}

\usepackage{nicefrac}
\usepackage[inline]{enumitem}

\usepackage{mathrsfs}

\usepackage{changepage} 

\usepackage{array,float}

\usepackage{tikz}
\usetikzlibrary{shapes,arrows}
\usetikzlibrary{fit,positioning}
\usetikzlibrary{patterns,decorations.pathreplacing}
\usetikzlibrary{decorations.pathmorphing,arrows.meta}
\usepackage{xkeyval}
\usepackage{moreverb}

\usepackage{epic}
\usepgfmodule{shapes,plot,decorations}

\usepackage{booktabs} 
\usepackage[normalem]{ulem}

\newcommand{\N}{\mathbb{N}}
\newcommand{\Z}{\mathbb{Z}}
\newcommand{\Q}{\mathbb{Q}}
\newcommand{\R}{\mathbb{R}}
\newcommand{\CC}{\mathbb{C}}

\newcommand{\HH}{\mathbb{H}}

\usepackage[warn]{mathtext}

\newcommand{\inj}{\operatorname{inj}}

\newcommand{\vol}{\operatorname{Vol}}

\newcommand{\Isom}{\mathrm{Isom}}

\usepackage{todonotes}

\newtheorem{theorem}{Theorem}[section]
\newtheorem{corollary}{Corollary}[theorem]
\newtheorem{lemma}[theorem]{Lemma}
\newtheorem{proposition}[theorem]{Proposition}
\newtheorem{definition}[theorem]{Definition}
\newtheorem{conjecture}[theorem]{Conjecture}
\newtheorem{question}[theorem]{Question}

\theoremstyle{remark}
\newtheorem{remark}[theorem]{Remark}

\theoremstyle{remark}

\renewcommand{\qed}{\hfill$\scriptstyle\blacksquare$}

\title{Random walks on cocompact Fuchsian and Kleinian groups}

\author{Nikolay Bogachev}
\address{Department of Computer and Mathematical Sciences, University of Toronto Scarborough, 1095 Military Trail, Toronto, ON M1C 1A3, Canada}
\email{n.bogachev@utoronto.ca}

\author{Petr Kosenko}
\address{Department of Mathematics, University of British Columbia, Canada}
\email{pkosenko@math.ubc.ca}

\author{Giulio Tiozzo}
\address{Department of Mathematics, University of Toronto, 40 St George Street, Toronto ON, M5S~2E4, Canada}
\email{tiozzo@math.utoronto.ca}

\begin{document}

\setcounter{tocdepth}{1}

\begin{abstract}
The question of the singularity at infinity of the hitting measure of random walks has a long history, originating from the work of Furstenberg in the 1960s. In 2011, Kaimanovich and Le Prince conjectured that the hitting measure of any finitely supported random walk on a discrete subgroup $\Gamma$ of $\mathrm{SL}_N(\R)$ is singular at infinity with respect to the Lebesgue measure. Using algebraic and geometric convergence and hyperbolic Dehn filling, we prove the singularity conjecture for certain measures on ``most'' cocompact Fuchsian and Kleinian groups. 
\end{abstract}

\subjclass[2020]{}

\maketitle

\section{Introduction}

This paper concerns hitting measures of finitely supported random walks on finitely generated discrete subgroups of isometries of hyperbolic spaces $\HH^d$. This subject is currently of growing interest due to the rich interplay between geometric group theory, geometric topology, probability and dynamics. One of the central questions in this field is the singularity conjecture of Kaimanovich and Le Prince~\cite{KLP11}, which can be broadly formulated for discrete subgroups $\Gamma < \mathrm{SL}_N(\R)$ as follows.

\begin{conjecture}[Kaimanovich and Le Prince \cite{KLP11}]\label{conj:sing}
    The hitting measure of any random walk finitely supported on a discrete subgroup $\Gamma < \mathrm{SL}_N(\R)$ is singular at infinity with respect to the Lebesgue measure.
\end{conjecture}

\noindent This conjecture has, in fact, a very long history originating from the work of Furstenberg \cite{Fur63, Fur71} who proved that for any lattice $\Gamma$ in a semisimple Lie group, there exists an infinitely-supported random walk on $\Gamma$ with finite first moment in the hyperbolic metric such that the hitting measure is absolutely continuous on the boundary.

The singularity conjecture remains widely open even for lattices in Lie groups, including $\mathrm{PSL}_2(\R)$, $\mathrm{PSL}_2(\CC)$, and $\mathrm{PO}_{d,1}(\R)$. We are primarily interested in {\em hyperbolic lattices} $\Gamma < \Isom(\HH^d)\cong\mathrm{PO}_{d,1}(\R)$. The orientation-preserving isometry group $\Isom^+(\HH^d)$ of $\HH^d$ can be identified with $\mathrm{PO}^\circ_{d,1}(\R)$, and in low dimensions we have $\Isom^+(\HH^2) \cong \mathrm{PSL}_2(\R)$ and $\Isom^+(\HH^3) \cong \mathrm{PSL}_2(\CC)$. Discrete subgroups of the latter two Lie groups are called {\em Fuchsian} and {\em Kleinian}, respectively. 

As mentioned above, even in the special case of hyperbolic lattices, Conjecture~\ref{conj:sing} is still open; however, some progress was obtained. The singularity conjecture is proven for {\em non-uniform} lattices in $\mathrm{PSL}_2(\R)$, see Guivarc'h--Le Jan~\cite{GLJ93}, Deroin--Kleptsyn--Navas~\cite{DKN09}, and Blachère--Haïssinsky--Mathieu~\cite{blachere2011harmonic}. Recently, Ran\-decker--Tiozzo~\cite{RT21} proved this conjecture for random walks on {\em all non-uniform} hyperbolic lattices $\Gamma < \mathrm{PO}_{d,1}(\R)$. A similar result for cocompact Fuchsian groups with centrally symmetric fundamental domains was obtained by Kosenko--Tiozzo \cite{KT22}. Not that much is known in higher rank either. For example, if $G$ is a semisimple Lie group with Kazhdan's property (T), $\Gamma < G$ is an infinite-covolume discrete Zariski-dense subgroup, and $\mu$ is a symmetric probability measure on $\Gamma$ such that $\langle\mathrm{supp}\,\mu\rangle = \Gamma$, then the Furstenberg measure $\nu$ on $G/P$ is singular with respect to the Lebesgue measure; see recent preprints Lee--Tiozzo--Van Limbeek~\cite{LTvL25} and Kim--Zimmer \cite{KZ25}. 

We prove the singularity conjecture for large families of cocompact Fuchsian and Kleinian groups. To formulate the main result, we need to introduce some preliminary information.

The key novel idea of this paper is to geometrically approximate groups for which the singularity conjecture is proven by a sequence of other groups. Our results and methods rely on the theory of {\em small deformations} and {\em geometric convergence} of finite-volume hyperbolic $3$-manifolds. Geometric convergence is equivalent to the convergence in the {\em Chabauty topology} of subgroups of $\mathrm{PSL}_2(\CC)$. On the other hand, small deformations mean the deformations of a certain representation and give rise to the {\em algebraic convergence} of representations of a given group. It is well known that among lattices in $\Isom(\HH^d)$, $d \ge 3$, only nonuniform lattices in $\Isom(\HH^3)$ admit small deformations (for uniform lattices it was proved by Weil \cite{Weil62}, while for nonuniform lattices it is a result of Garland--Raghunathan~\cite{GR70}). 

Nonuniform lattices in $\Isom(\HH^3)$ are usually not locally rigid, meaning that their small deformations are not conjugate to the initial representation, as was discovered by Thurston. The small deformations we discuss here do not lead, in general, to lattices or even discrete subgroups of $\Isom(\HH^3)$. However, Thurston's {\em hyperbolic Dehn filling} theorem allows one to obtain lattices under certain parameters of such deformations.  We discuss the relation between small deformations, geometric convergence, and hyperbolic Dehn filling in Section \ref{sec:alg-geom-conv}. Moreover, we note that for any geometrically convergent sequence of lattices $\Gamma_n \to \Gamma$ one can always choose measures $\mu_n$ on $\Gamma_n$ and $\mu$ on $\Gamma$ which ``are compatible with each other''. This compatibility comes from the hyperbolic Dehn filling on the limit lattice $\Gamma$; for a precise definition, see Section \ref{sec:compatible}. Here, we just highlight the fact that picking a single measure $\mu$ on $\Gamma$ is enough to ``translate'' it to measures $\mu_n$ on all but finitely many sequence elements $\Gamma_n$ such that all $\mu_n$ and $\mu$ are pairwise compatible.
Finally, a measure $\mu$ on a lattice $\Gamma$ is non-degenerate if $\Gamma$ is the group generated by the support of $\mu$.

\begin{theorem}\label{th:main-Kleinian} 
    For any geometrically convergent sequence $\Gamma_n \to \Gamma$ of torsion-free lattices $\Gamma_n, \Gamma < \mathrm{PSL}_2(\CC)$, where $\Gamma_n$ are not pairwise conjugate, equipped with any sequence of compatible, non-degenerate, finitely-supported measures $\mu_n$ and $\mu$, respectively, there exists a number $N = N(\mu) > 0$ such that for all $n>N$ the hitting measure $\nu_{\mu_n}$ of the random walk on $\Gamma_n$ driven by $\mu_n$ is singular with respect to the Lebesgue measure on the ideal boundary $\partial_\infty \HH^3 \cong \mathbb{S}^2$.
\end{theorem}

In Section~\ref{sec:orbifolds}, we explain how the tools of geometric convergence and hyperbolic Dehn filling work for compact hyperbolic $3$-orbifolds, i.e. for cocompact Kleinian groups not necessarily torsion-free. The orbifold version of Theorem \ref{th:main-Kleinian} is the following.

\begin{theorem}\label{th:main-3-orbifold}
    Let $\Gamma < \mathrm{PSL}_2(\CC)$ be a non-uniform lattice such that each cusp of the (oriented) orbifold $\mathcal{O}=\HH^3/\Gamma$ is either a torus $\mathbb{T}^2$ or a pillowcase $S^2(2,2,2,2)$. Consider a sequence of lattices $\Gamma_n$ that geometrically converge to $\Gamma$. Then for any non-degenerate finitely supported measure $\mu$ on $\Gamma$ there is $N > 0$ such that for all $n>N$ the following conditions hold:
    \begin{enumerate}
        \item The orbifolds $\mathcal{O}_n = \HH^3/\Gamma_n$ are obtained by hyperbolic Dehn filling of one of the cusps of $\mathcal{O}$.
        \item There exist finitely supported measures $\mu_n$ on $\Gamma_n$ compatible with the measure $\mu$ on $\Gamma$.
        \item The hitting measure $\nu_{\mu_n}$ of the random walk driven by $\mu_n$ is singular with respect to the Lebesgue measure on the ideal boundary $\partial_\infty \HH^3 \cong \mathbb{S}^2$.
    \end{enumerate}
\end{theorem}
\begin{proof}
    Follows from Section \ref{sec:orbifolds}, including the hyperbolic Dehn filling theorem for orbifolds by Dunbar--Meyerhoff \cite{DM94}, Proposition \ref{prop:entropy-inequality}, and from the last part of the proof of Theorem \ref{th:main-Kleinian}.
\end{proof}

Here we provide a special but very natural version of our main theorem in the setting of hyperbolic reflection groups.

\begin{theorem}\label{th:main-Coxeter}
    Let $\Sigma$ be the Coxeter--Vinberg diagram on $m$ vertices of a compact Coxeter polyhedron $P \subset \HH^3$ with a dihedral angle $\pi/n$, $n \ge 6$, at one of its edges $e$. By varying this parameter $n$ with other dihedral angles remaining fixed, we obtain a sequence of compact Coxeter polyhedra $P_n \subset \HH^3$ whose corresponding cocompact reflection groups $\Gamma_n$ (generated by reflections $r^{(n)}_1, \ldots, r^{(n)}_m$ in the walls of $P_n$) satisfy the following conditions.
    \begin{enumerate}
        \item $\Gamma_n \to \Gamma$ geometrically, where $\Gamma$ is a cofinite group generated by reflections $r_1, \ldots, r_m$ in the walls of a $1$-cusped polyhedron $P_\infty$ obtained from $P_n$ by contracting the edge $e$ to a point $v_\infty$ on the ideal boundary $\partial \HH^3$.
        \item Assume that the measures $\mu_{n}$ are 
        defined by assigning probability $p_i >0$  to the generating reflection $r^{(n)}_i$, where $1 \le i \le m$ and the $p_i$ do not depend on $n$. Then there exists a number $N > 0$ such that for all $n>N$ the hitting measure $\nu_{\mu_n}$ of the finitely supported random walk  on $\Gamma_n$ with distribution $\mu_n$ is singular with respect to the Lebesgue measure on the ideal boundary $\partial_\infty \HH^3 \cong \mathbb{S}^2$.
    \end{enumerate}
\end{theorem}

In the case of Fuchsian groups, the situation is different. In contrast to the Kleinian groups case, there are, in general, uncountably many deformations of lattices in $\Isom(\HH^2)$. Moreover, the algebraic and geometric convergence do not coincide even in the case of lattices in $\Isom(\HH^2)$. See Section \ref{sec:examples} for examples.

\begin{theorem}\label{th:main-Fuchsian}
Let $\Gamma$ be a non-uniform lattice in $\mathrm{PSL}_{2}(\R) = \mathrm{Isom}^+(\HH^2)$, i.e. a cofinite Fuchsian group, and $\mu$ a non-degenerate finitely supported measure on $\Gamma$. Then for every sequence of cocompact Fuchsian groups $\Gamma_n=\rho_n(\Gamma)$ algebraically converging to $\Gamma=\rho(\Gamma)$, there exists $N \in \N$ such that for every $n > N$, the hitting measure $\nu_n = \nu_{\mu_n}$ of $\mu_n := (\rho_n)_*\mu$ is singular with respect to the Lebesgue measure on $\partial_\infty \HH^2 \cong \mathbb{S}^1$.
\end{theorem}

We also note that there exist many small deformations of cocompact Fuchsian groups that are also cocompact Fuchsian. This allows one to extend Theorem \ref{th:main-Fuchsian} to uniform lattices in $\mathrm{PSL}_{2}(\R)$ that algebraically approximate those for which this theorem is proven, and one can thus continue adding more and more groups that satisfy the singularity conjecture. 

\subsection*{Funding}  N.B. was supported by NSERC Discovery Grant RGPIN-2024-05680. P.K.'s research is supported by the Natural Sciences and Engineering Research Council of Canada (NSERC). G.T. was supported by NSERC Discovery Grant RGPIN-2024-04324.

\subsection*{Acknowledgements} We thank Nathan Dunfield, David Futer, Misha Kapovich, and Jean Raimbault for very fruitful discussions on Dehn surgery and geometric convergence. We are grateful to Misha Kapovich for suggesting the proof outline for Theorem \ref{th:geom-alg-conv}, part (1), and to Fedor Vylegzhanin for his help with algebraic topology.

\section{Preliminaries}\label{sec:prelim}

\subsection{Hyperbolic manifolds and orbifolds}
The  $d$-dimensional {\em hyperbolic space} $\HH^d$ is the unique connected, simply connected Riemannian manifold of constant negative curvature. It has several useful models, among which we are mostly interested in the {\em hyperboloid model} $\mathcal{H}^d$ and the {\em upper half-space model} $\mathcal{U}^d$. 

To describe the hyperboloid model, we consider the Minkowski space $\R^{d,1}$, i.e. $\R^{d+1}$ equipped with a Lorentzian scalar product $(\cdot, \cdot)$ of signature $(d,1)$. Passing to the standard basis of $\R^{d,1}$, we can write
$(x,y)=-x_0 y_0 + x_1 y_1 + \ldots + x_d y_d$. 
Then we set
$$
\mathcal{H}^d := \{x \in \R^{d,1}\mid (x,x)=-1, x_0 > 0\},
$$
which is a connected component of the hyperboloid $\{x \in \R^{d,1} \mid(x,x)=-1\}$. The Riemannian metric tensor $g_{\mathcal H}$ is defined as the restriction of the Lorentzian scalar product $(\cdot, \cdot)$ on the tangent bundle $T \mathcal{H}^d$.
The hyperbolic distance function $\rho_{\mathcal{H}}$ then satisfies $\cosh \rho_{\mathcal{H}}(x,y)=-(x,y)$. The full isometry group $\Isom(\HH^d)$ is isomorphic to the index-$2$ subgroup $\mathrm{O}'_{d,1}(\R)$ of $\mathrm{O}_{d,1}(\R)$ defined by the condition that every matrix $A \in \mathrm{O}'_{d,1}(\R)$ leaves $\mathcal{H}^d$ invariant. This group can also be identified with the projective orthogonal group $\mathrm{PO}_{d,1}(\R) = \mathrm{O}_{d,1}(\R)/\{\pm \mathtt{id}\}$. Totally geodesic lines (geodesics) and planes are the intersections of $\mathcal{H}^d$ with vector subspaces of $\R^{d,1}$. Hyperbolic reflections in hyperplanes are linear reflections with respect to hyperplanes of $\R^{d,1}$.

The upper half-space model $\mathcal{U}^d$ is defined as follows:
$$
\mathcal{U}^d = \{x=(x_1,\ldots,x_d) \in \R^d \mid x_d > 0\} \text{ with a metric tensor } g_{\mathcal{U}} = \frac{g_\mathcal{E}}{x_d^2}, 
$$
where $g_{\mathcal{E}}$ is the standard Euclidean metric on $\R^d$. Totally geodesic planes (including one-dimensional lines) are, in this model, either vertical Euclidean half-planes or hemispheres orthogonal to the boundary hyperplane $\{x\in \R^d \mid x_d = 0\}$.

In dimensions $2$ and $3$, it is useful to involve the complex plane in this model: 
$$
\mathcal{U}^2 = \{z \in \CC\mid \mathtt{Im}(z)>0\} \text{ and } \mathcal{U}^3 = \{(z,t) \in \CC\times \R_{>0}\}.$$ The orientation-preserving isometry groups of $\HH^2$ and $\HH^3$ are then identified with the groups of fractional linear (Möbius) automorphisms of $\mathcal{U}^2$ and $\CC$, respectively: 
$\Isom^+(\HH^2) \cong \mathrm{PSL}_2(\R)$ and $\Isom^+(\HH^3) \cong \mathrm{PSL}_2(\CC)$.

Any complete connected {\em hyperbolic $d$-manifold} is a quotient $M=\HH^d/\Gamma$, where $\Gamma < \Isom(\HH^d)$ is a torsion-free discrete subgroup. If $\Gamma$ has finite-order elements, then $M$ is called a {\em hyperbolic orbifold} and $\Gamma$ its {\em orbifold fundamental group}. The manifold (or orbifold) $M$ has finite volume if and only if $\Gamma$ is a lattice in the real Lie group $\Isom(\HH^d)$, i.e. $\Gamma$ has finite-volume fundamental domain in $\HH^d$. The orbifold $M$ is closed if $\Gamma$ is a cocompact (or uniform), i.e. if its fundamental domain is compact in $\HH^d$.

\subsection{Thick-thin decomposition}

First, we will need some fundamental notions from Riemannian geometry.

It is well known that in any Riemannian manifold $(M,g)$, given a point $p \in M$ and a tangent vector $T_p M$, there is a unique maximal geodesic $\gamma_v \colon I \to M$, where $I \subset \R$ is some interval containing $0$, such that $\gamma(0)=p$ and $\gamma'(0)=v$. This allows us to define the {\em exponential map} from the tangent bundle $TM$ to $M$ as follows. 

\begin{definition}
Let $\mathcal{E} = \{v \in TM \mid \gamma_v \text{ is defined on } I \supset [0,1]\}$. Then the exponential map $\exp : \mathcal{E} \to M$ is defined by $\exp(v) = \gamma_v(1).$
\end{definition}

The exponential map is smooth, its differential at the origin of each $T_p M$ is the identity map, whence $\exp_p \colon \mathcal{E}_p \to M$, where $\mathcal{E}_p = \mathcal{E}\cap T_p M$, is a local diffeomorphism at the origin of $T_p M$. Now we can define the {\em injectivity radius}.  

\begin{definition}
    Let $(M,g)$ be a Riemannian manifold. The injectivity radius $\mathrm{inj}_p M$ of $M$ at a point $p \in M$ is the supremum of all $r > 0$ such that $B(0,r) \subset \mathcal{E}_p$ and $\exp_p\mid_{B(0,r)}$ is a diffeomorphism onto its image.
\end{definition}

Let us now return to the setting of complete hyperbolic manifolds $M=\HH^d/\Gamma$, possibly of infinite volume. Since $\Gamma$ acts properly discontinuously by isometries on the universal cover $\HH^d$ of $M$, one can reach the minimal distance between the preimages of $p \in M$ in $\HH^d$: this minimal distance $D$ is equal to the length of the shortest homotopically non-trivial loop at $p$. Moreover, $\mathrm{inj}_p M = D/2$. More precisely, denoting by $\mathrm{disp}(\gamma)$ the minimal displacement of $\gamma \in \Gamma$, we have
$$
\mathrm{inj}_p M = \frac{1}{2}\inf\{\mathrm{disp}(\gamma)\mid \gamma \in \Gamma - \{e\}\}.
$$
Therefore, the injectivity radius of $p \in M$ along a closed simple geodesic of length $\varepsilon$ is precisely $\varepsilon/2$.

The following group-theoretic result, known as the Margulis Lemma, has broad applications in geometry and topology of hyperbolic manifolds.

\begin{lemma}[Margulis Lemma]\label{lem:Margulis}
    For every dimension $d \geq 2$, there is a constant $\varepsilon_d > 0$ such that for all $x \in \HH^d$, every discrete group $\Gamma < \Isom(\HH^d)$ generated by elements that move $x$ at distance $<\varepsilon_d$ is virtually abelian.
\end{lemma}

\begin{definition}
The constant $\varepsilon_d>0$ from Lemma \ref{lem:Margulis} is called the Margulis constant.   
\end{definition}
 
\begin{definition}[Thick-Thin Decomposition]\label{def:thick-thin}
Fix $d\ge 2$.  For any complete hyperbolic $d$–manifold $M$ and any $\varepsilon>0$, we define the $\varepsilon$–thick part
$$
M^{\ge \varepsilon} = \{p\in M \mid \inj_p(M) \ge \varepsilon/2\}
$$
and the $\varepsilon$–thin part
$$
M^{\le \varepsilon} = \mathrm{clos}(M - M^{\ge \varepsilon}).
$$
\end{definition}

Note that if there is a closed geodesic $\alpha$ of length precisely $\varepsilon$, then $\alpha \subset M^{\ge \varepsilon}$ but $\alpha \not\subset M^{\le \varepsilon}$.

The central result on the thick-thin decomposition states (see \cite[Theorem 4.2.14]{Mar}) that if $0<\varepsilon \le \varepsilon_d$ (the Margulis constant), then the thin part $M^{\le \varepsilon}$ of any complete hyperbolic $d$-manifold $M$ is a disjoint union of star-shaped neighbourhoods of cusps and of simple closed geodesics of length $< \varepsilon$. In fact, all closed geodesics of length $< \varepsilon_d$ are simple and disjoint. If $d \le 3$ and $M$ is orientable, then the thin part $M^{\le \varepsilon}$ consists of truncated cusps and $R$-tubes; see \cite[Proposition 4.2.15]{Mar}. 

The following fundamental fact (see \cite[Proposition 4.2.17 and Corollary 4.2.18]{Mar} or \cite[Proposition D.2.6 and Corollary D.3.14]{BP-lectures}) about the topology of finite-volume hyperbolic manifolds plays an important role in this paper.

\begin{proposition}\label{prop:thick-thin-finite-volume}
    Let $0<\varepsilon \le \varepsilon_d$. A complete hyperbolic $d$-manifold $M$ has finite volume if and only if its thick part $M^{\ge \varepsilon}$ is compact. Moreover, every complete finite-volume hyperbolic manifold $M$ is diffeomorphic to the interior of a compact manifold with boundary that consists of closed flat manifolds. In particular, any orientable hyperbolic $3$-manifold of finite volume can be identified with a compact $3$-manifold with boundary consisting of (finitely many) flat tori.
\end{proposition}

\begin{remark}\label{rem:thin-small}
It is known (see \cite[Proposition 4.3.1]{Mar}) that, given a finite-volume hyperbolic manifold $M$ and any number $L>0$, there are only finitely many closed geodesics in $M$ shorter than $L$. This allows us to choose $\varepsilon$ to be small enough (namely, smaller than half the systole of $M$) so that the thin part $M^{\le \varepsilon}$ consists only of the cusp neighbourhoods and the thick part $M^{\ge \varepsilon}$ is a compact manifold with finitely many flat boundary components, as shows Proposition \ref{prop:thick-thin-finite-volume}. Such a thick part $M^{\ge \varepsilon}$ is homotopy equivalent to $M$ and is thus an aspherical manifold.
\end{remark}

\subsection{Deformations, convergence, and hyperbolic Dehn filling}\label{sec:def}

\subsubsection{Deformations and convergence: general theory}

Let $G$ be a (connected) Lie group, and $\Gamma$ an abstract group. Denote by $\mathrm{Hom} (\Gamma, G)$ the set of homomorphisms of the group $\Gamma$ to the group $G$ with the pointwise convergence topology, that is, a sequence $\rho_n \in \mathrm{Hom} (\Gamma, G)$ is said to {\em algebraically converge} to a homomorphism $\rho$ if $\lim_{n \to \infty} \rho_n(\gamma) = \rho(\gamma)$ for every $\gamma \in \Gamma$.

Let $\Gamma_n$ be a sequence of discrete subgroups of $\Isom(\HH^d)$. We say that $\Gamma_n$ \emph{converge in the Chabauty topology} to $\Gamma$ if the groups converge in the Chabauty topology on closed subsets of $G$. Convergence in this topology can be characterized as follows:
\begin{itemize}
  \item every $\gamma \in \Gamma$ is the limit of some sequence $\{\gamma_n\}$ with $\gamma_n \in \Gamma_n$;
  \item if $\gamma_n \to \gamma$ is a convergent sequence with $\gamma_n \in \Gamma_n$, then $\gamma \in \Gamma$.
\end{itemize}

Let $\mathscr{H}_d$ denote the set of complete hyperbolic manifolds $\HH^d/\Gamma$, with $\Gamma < \Isom(\HH^d)$ a discrete subgroup, and $\mathscr{F}_d \subset \mathscr{H}_d$ the subset of finite-volume manifolds (i.e. with $\Gamma$ being a lattice). We equip $\mathscr{H}_d$ with the topology of {\em geometric convergence} defined as the Chabauty topology on $\Isom(\HH^d)$. That is, hyperbolic manifolds $M_n=\HH^d/\Gamma_n$ converge to $M = \HH^d/\Gamma$ geometrically if $\Gamma_n \to \Gamma$ in the Chabauty topology.

We will need another notion of convergence which happens to be equivalent to the geometric one in the setting of hyperbolic manifolds. For a manifold $X$, let $\mathcal{OF}(X)$ denote the orthogonal frame bundle of $X$. Let $(M_1,v_1)$ and $(M_2,v_2)$ be complete hyperbolic $d$–manifolds with orthonormal baseframes $v_1 \in \mathcal{OF}(M_1)$, $v_2 \in \mathcal{OF}(M_2)$ based at points $x_1\in M_1$, $x_2\in M_2$. 
We say that a map $f:(X_1,v_1)\to (X_2,v_2)$ is a \emph{framed $(K,r)$–approximate isometry} between $(M_1,v_1)$ and $(M_2,v_2)$ if:
\begin{itemize}
  \item $f:X_1\to X_2$ is a diffeomorphism of domains $X_j\subset M_j$;
  \item $B_{M_1}(x_1,r)\subset (X_1,x_1)\subset (M_1,x_1)$ and 
        $B_{M_2}(x_2,r)\subset (X_2,x_2)\subset (M_2,x_2)$;
  \item $Df(v_1)=v_2$ (the baseframe is preserved);
  \item for all $x,y\in X_1$,
  $$
     \frac{1}{K}\, d_{M_1}(x,y) \;\le\; d_{M_2}\bigl(f(x),f(y)\bigr) \;\le\; K\, d_{M_1}(x,y).
  $$
\end{itemize}

Suppose now that $M_i=\HH^d/\Gamma_i$ is a sequence of hyperbolic manifolds with baseframes $v_i \in \mathcal{OF}(M_i)$ based at points $x_i \in M_i$, and let $M=\HH^d/\Gamma$ be a hyperbolic manifold with baseframe $v \in \mathcal{OF}(M)$ based at $x\in M$. The sequence $(M_i, v_i)$ converges to $(M,v)$ in the induced {\em refined Gromov--Hausdorff topology} if if there exist $(K_i,r_i)$–approximate isometries $f_i:(M_i,v_i)\to (M,v)$ with $K_i\to 1$ and $r_i\to\infty$ as $i\to\infty$.

The following theorem (see Benedetti--Petronio \cite[Theorem E.1.13 and Remark E.1.19]{BP-lectures} and Canary--Epstein--Green \cite[Theorem 3.2.9 and Corollary 3.2.11]{CEG87}) shows that the geometric convergence is equivalent to the classical pointed Gromov--Hausdorff convergence and the stronger convergence in the refined Gromov--Haus\-dorff topology, as well as to the pointed Gromov--Hausdorff convergence of the corresponding Dirichlet fundamental polyhedra.

\begin{theorem}\label{th:equivalent-convergences}
    Let $d \ge 2$. The following are equivalent.
    \begin{enumerate}
        \item A sequence of torsion-free discrete subgroups $\Gamma_n < \Isom(\HH^d)$ geometrically converges to a torsion-free discrete subgroup $\Gamma < \Isom(\HH^d)$.
        \item The pointed hyperbolic manifolds $(\HH^d/\Gamma_n, v_n)$ converge to $(\HH^d/\Gamma, v)$ in the refined Gromov--Hausdorff topology for some choice of orthonormal baseframes $v_n \in \mathcal{OF}(\HH^d/\Gamma_n)$ and $v \in \mathcal{OF}(\HH^d/\Gamma)$.
        \item For a fixed origin $o \in \HH^d$, the Dirichlet fundamental polyhedra $D_{\Gamma_n}(o)$ converge in the pointed Gromov--Hausdorff topology to the Dirichlet fundamental polyhedron $D_{\Gamma}(o)$. The corresponding side-pairing generators of $\Gamma_n$ also converge to that of $\Gamma$.
    \end{enumerate}
\end{theorem}

\begin{corollary}\label{cor:phi-n}
    If $\Gamma_n \to \Gamma$ geometrically, then for any compact subset $K \subset M=\HH^d/\Gamma$ there is $N>0$ such that for any $n>N$ there exists a $(1+\varepsilon_n)$-bi-Lipschitz embedding $\phi_n : K \hookrightarrow M_n=\HH^d/\Gamma_n$, which is a diffeomorphism onto its image $\phi_n(K)$.
\end{corollary}
\begin{proof}
Suppose that $K \subset M$ is obtained from a compact lift $\widetilde{K} \subset \HH^d$ by projection $\HH^d \to \HH^d/\Gamma$. One can pick such $\widetilde{K}$ in the Dirichlet fundamental polyhedron $D_{\Gamma}(o)$ for some point $o \in \HH^d$.  Then there is $R>0$ such that $\widetilde{K} \subset B_{\HH^d}(o, R)$. Theorem~\ref{th:equivalent-convergences} implies that there exist $(K_n,r_n)$–approximate isometries $f_n:(M_n,v_n)\to (M,v)$ with $K_n\to 1$ and $r_n\to\infty$ as $n\to\infty$. This means that there is $N>0$ such that for any $n>N$ we have $r_n > R$ and a diffeomorphism of domains in $M_n$ and $M$, where the selected domain in $M$ covers the ball $B_{M}(x,r_n)$.
\end{proof}

\subsubsection{Hyperbolic Dehn filling}

Let $M$ be a $3$-manifold with the (connected) torus boundary $T = \partial M$.

\begin{definition}[Dehn filling]
A Dehn filling of $M$ along $T$ is the operation of gluing a solid torus $D \times S^1$ to $M$ via a diffeomorphism $\varphi: \partial D \times S^1 \to T$.
\end{definition}

The closed curve $\partial D \times \{x\}$ is glued to some simple closed curve $\gamma \subset T$. The result of this operation is a closed manifold $M'$. It is well known (see, for instance, \cite[Proposition 10.1.3]{Mar}) that
$$
\pi_1 (M')=\pi_1(M)/\langle\!\langle\gamma\rangle\!\rangle, 
$$
where $\langle\!\langle\gamma\rangle\!\rangle$ denotes the normal closure of $\gamma$ in $\pi_1(M)$. This gives a surjective homomorphism $f: \pi_1(M) \to \pi_1 (M')$ with kernel $\ker f = \langle\!\langle\gamma\rangle\!\rangle$.

We will say that our Dehn filling {\em kills} the curve $\gamma.$ Fixing a basis $\alpha, \beta$ for $\pi_1(T) < \Gamma$, we can write isotopy classes (called {\em slopes}) of homotopically non-trivial unoriented simple closed curves on $T$ as $[\gamma]=\alpha^p \beta^q$ for some coprime pair $(p,q)$. This gives a 1-1 correspondence between slopes and elements of $\Q \cup \{\infty\}$. Therefore, every number $\frac{p}{q}$ corresponds to a Dehn filling that kills an associated slope~$[\gamma]$.

Now we need to understand what happens with hyperbolic structures when we perform a Dehn filling procedure. Recall that by Proposition \ref{prop:thick-thin-finite-volume} any oriented finite-volume hyperbolic $3$-manifold $M = \HH^3/\Gamma$ with one cusp is diffeomorphic to a $3$-manifold with the torus boundary. The following theorem shows that almost all Dehn fillings $M'$ obtained from $M$ admit complete hyperbolic structures and therefore give rise to closed hyperbolic $3$-manifolds.

\begin{theorem}[Thurston's hyperbolic Dehn filling]\label{th:Thurston-Dehn}
    Let $M$ be a finite-volume hyperbolic $3$-manifold with one cusp. Let $\mathscr{M}$ be the set of all $3$-manifolds obtained from $M$ via Dehn filling. Then all but finitely many elements of $\mathscr{M}$ are closed hyperbolic $3$-manifolds, and $\mathscr{M}$ contains sequences $\{M_n\}$ converging to $M$.
\end{theorem}
\begin{proof}
    For a complete proof of the first part of the theorem we refer to Martelli \cite[Theorem 15.1.1 and Corollary 15.1.3]{Mar}. The second (convergence) part can be found in Benedetti--Petronio \cite[Theorem E.5.1]{BP-lectures}.
\end{proof}

\begin{remark}\label{rem:Dehn-deformation}
Denote by $\rho \colon \Gamma \to \Isom(\HH^3) \cong \mathrm{PSL}_2(\CC)$ the initial embedding of $\pi_1 (M) = \Gamma$ in $\mathrm{PSL}_2(\CC)$. Then the above discussed homomorphism $f_{p,q}= f: \pi_1(M) \to \pi_1 (M')$ leads to a deformation of the representation $\rho$ into a (non-faithful) representation $\rho_{p,q}: \Gamma \to \mathrm{PSL}_2(\CC)$, where $\rho_{p,q} = \rho' \circ f$ with $\rho': \pi_1(M') \to \mathrm{PSL}_2(\CC)$ being a faithful representation of $\pi_1(M')$ obtained in Theorem~\ref{th:Thurston-Dehn}.    
\end{remark}

\subsection{Hyperbolic Dehn filling of orientable hyperbolic $3$-orbifolds}\label{sec:orbifolds} In the previous subsection, we discussed the hyperbolic Dehn filling for manifolds, however, this theory can be adapted to the broader setting of hyperbolic $3$-orbifolds, which was done by Dunbar and Meyerhoff \cite{DM94}. 

Let $\mathcal{O}$ be a complete, finite-volume, orientable hyperbolic $3$-orbifold with at least one cusp. Each cusp admits a horospherical cross-section which is a closed Euclidean $2$-orbifold $\mathcal F$. Thus, $\mathcal F$ is either a nonsingular torus, a pillowcase (see Figure \ref{fig:pillowcase}), or a turnover of
type $S^2(n_1, n_2, n_3)$ with $\frac{1}{n_1}+\frac{1}{n_2}+\frac{1}{n_3}=1$:
$$
\mathbb{T}^2,\quad S^2(2,2,2,2),\quad S^2(3,3,3),\quad S^2(2,3,6),\quad S^2(2,4,4).
$$
We recall that $S^2(a_1,\dots,a_k)$ denotes an orbifold with underlying topological space $\mathbb{S}^2$ and cone points of orders $a_1,\dots,a_k$.

Let $\Gamma = \pi_1^{\mathrm{orb}}(\mathcal{O})$, and let $\Gamma_u < \Gamma$ be a maximal parabolic subgroup corresponding to a cusp; $\Gamma_u$ is precisely the stabilizer of a point $u \in \partial \mathbb{H}^3$. The cusp cross-section $\mathcal F$ is naturally identified with the Euclidean orbifold $\mathbb{E}^2 / \Gamma_u$, and $\Gamma_u$ fits into a short exact sequence
$$
1 \longrightarrow \Lambda \cong \mathbb{Z}^2 \longrightarrow \Gamma_u \longrightarrow H \longrightarrow 1,
$$
where $\Lambda$ is the subgroup of translations in $\Gamma_u$ and $H$ is a finite subgroup of $\mathrm{O}_2(\R)$, the rotational holonomy of $\mathcal F$.

A cusp of $\mathcal{O}$ is called \emph{non-rigid} (or \emph{flexible}) if its cross-section is $\mathbb{T}^2$ or the pillowcase $S^2(2,2,2,2)$, and \emph{rigid} if its cross-section is $S^2(3,3,3)$, $S^2(2,3,6)$ or $S^2(2,4,4)$.  Equivalently, $F$ is non-rigid if and only if it admits a nontrivial $1$-parameter family of Euclidean similarity structures, i.e. if its Teichmüller space has positive dimension, and rigid otherwise.

For a non-rigid cusp, the subgroup of translations $\Lambda \cong \mathbb{Z}^2$ in $\Gamma_p$ is well defined up to conjugacy. A \emph{slope} on the cusp is, by definition, an unoriented isotopy class of essential simple closed curves on $F$ in the orbifold sense. Group-theoretically, a slope determines a primitive element $\lambda \in \Lambda$, well defined up to inversion and the action of the finite holonomy $H$. One can therefore regard the set of slopes as the set of $H$-orbits of primitive elements in $\Lambda$ modulo $\lambda \sim \lambda^{-1}$.

Topological Dehn filling on a torus cusp is just the usual operation from the manifold case: one glues a solid torus along the boundary component, identifying a chosen slope on $\mathbb{T}^2$ with a meridian of the solid torus. In the orbifold setting one may more generally perform \emph{orbifold Dehn filling} by gluing a solid toric orbifold whose core curve is a cone geodesic of order $m \in \mathbb{N}$; the case $m=1$ recovers manifold Dehn filling.

For a cusp with cross-section $S^2(2,2,2,2)$, there is an analogous operation where one glues in a ``solid pillow,'' i.e.\ a $3$-ball with two unknotted singular arcs labelled by $\mathbb{Z}/2\mathbb{Z}$. This construction can be regarded as the quotient of the torus case by the elliptic involution on a $2$-fold torus cover of the cusp.  In this setting, a slope is represented by a simple closed curve on $S^2(2,2,2,2)$ which separates the four order-$2$ cone points into two pairs, and Dehn filling identifies this curve with a meridian of the solid pillow.

In contrast, there is \emph{no} nontrivial Dehn filling operation on a rigid cusp: no quotient of a solid torus has a Euclidean turnover as boundary, so one cannot topologically cap off a turnover boundary component by a solid toric orbifold.  Thus, Dehn filling is only defined on cusps whose cross-sections are $\mathbb{T}^2$ or $S^2(2,2,2,2)$.

Fix a non-rigid cusp of $\mathcal{O}$ with peripheral subgroup $\Gamma_u < \Gamma$ and translation subgroup $\Lambda \cong \mathbb{Z}^2$. Let $\lambda \in \Lambda$ be a primitive element representing a slope on the cusp in the sense above.

\begin{itemize}
  \item If the cusp cross-section is a torus, $\mathcal F = \mathbb T^2$, then $\Gamma_u = \Lambda \cong \mathbb{Z}^2$ and we may choose generators $a,b$ for $\Gamma_u$ so that $\lambda = a^p b^q$ for some coprime integers $p,q$.

  \item If the cross-section is the pillowcase $\mathcal F = S^2(2,2,2,2)$, then we have the short exact sequence
  $$
  1 \longrightarrow \Lambda \cong \mathbb{Z}^2 \longrightarrow \Gamma_u \longrightarrow \mathbb{Z}/2\mathbb{Z} \longrightarrow 1,
  $$
  and one may choose a presentation
  $$
  \Gamma_u = \langle a,b,t \mid [a,b]=1,\ t^2 = 1,\ tat^{-1} = a^{-1},\ tbt^{-1} = b^{-1} \rangle,
  $$
  where $\langle a,b\rangle \cong \mathbb{Z}^2$ is the subgroup of translations and $t$ is the elliptic involution. A slope is again represented by a primitive element $\lambda = a^p b^q \in \Lambda$.
\end{itemize}

Given $\lambda$ and an integer $m \geq 1$, the \emph{orbifold Dehn filling of slope $\lambda$ and cone order $m$} is, at the level of orbifold fundamental groups, the quotient
$$
\Gamma(\lambda;m) := \Gamma \big/ \langle\!\langle \lambda^m \rangle\!\rangle,
$$
where $\langle\!\langle \cdot \rangle\!\rangle$ denotes normal closure in $\Gamma$. When $\mathcal F = \mathbb{T}^2$ and $m=1$, this recovers the usual manifold Dehn filling, and for general $m$ it corresponds to turning the core curve of the attached solid torus into a cone geodesic of order $m$. When $\mathcal F = S^2(2,2,2,2)$ this describes the case of gluing a solid pillow whose core circle has cone order $m$.

The analytic and geometric theory of Dehn filling extends to orbifolds almost verbatim from the manifold case.  Fix a finite-volume orientable hyperbolic $3$-orbifold $\mathcal{O}$ with at least one non-rigid cusp. By the hyperbolic Dehn surgery theorem (see Dunbar--Meyerhoff \cite[Theorem 5.3]{DM94}) there exists a finite set of slopes $\mathcal{E}$ on the non-rigid cusps of $\mathcal{O}$ such that the following holds. For each choice of slopes $\lambda_i$ on the non-rigid cusps of $\mathcal{O}$ with $\lambda_i \notin \mathcal{E}$, and each choice of sufficiently large cone orders $m_i \in \mathbb{N}$, the orbifold $\mathcal{O}_{\lambda_i; m_i}$ obtained by topological orbifold Dehn filling along the $(\lambda_i;m_i)$ admits a unique complete hyperbolic orbifold structure. Moreover, its orbifold fundamental group is naturally isomorphic to the quotient of $\Gamma$ obtained by adding the relations $\lambda_i^{m_i} = 1$, and as the lengths of the filling slopes (and the cone orders $m_i$) tend to infinity, the resulting hyperbolic orbifolds $\mathcal{O}_{\lambda_i; m_i}$ converge geometrically to $\mathcal{O}$, see Dunbar--Meyerhoff \cite[Theorems 5.4 and 6.4]{DM94}.

In particular, every closed orientable hyperbolic $3$-orbifold arises from a cusped orientable hyperbolic $3$-orbifold by a (possibly multi-cusped) orbifold Dehn filling on non-rigid cusps, exactly mirroring the manifold setting; see \cite[Theorem 5.5]{DM94}.

\begin{figure}
    \centering
    \begin{tikzpicture}[scale=0.9,>=latex]

\def\R{2.2}   
\def\r{1.0}   

\draw (0,0) circle (\R);
\draw (0,0) circle (\r);

\draw[very thick] (0,-\R) arc (-90:90:\R);
\draw[very thick] (0,-\r) arc (-90:90:\r);

% axis
\draw[dash pattern=on 8pt off 4pt] (0,-3) -- (0,3);

\fill (0,\R)   circle (0.065);
\fill (0,\r)   circle (0.065);
\fill (0,-\r)  circle (0.065);
\fill (0,-\R)  circle (0.065);

\node[right]  at (0.05,\R+0.3)   {$\tilde a$};
\node[left]  at (-0.05,\r-0.3)  {$\tilde b$};
\node[right]  at (0.05,-\r+0.3) {$\tilde c$};
\node[left]  at (-0.05,-\R-0.3)  {$\tilde d$};

\pgfmathsetmacro{\ymer}{0.5*(\R+\r)}   
\pgfmathsetmacro{\bmer}{0.5*(\R-\r)}

% top vertical ellipse
\begin{scope}
  \clip (0,0) circle (\R);
  \clip (-3,-3) rectangle (0,3);
  \draw[dashed, very thick] (0,\ymer) ellipse (0.35 and \bmer);
\end{scope}

\begin{scope}
  \clip (0,0) circle (\R);
  \clip (0,-3) rectangle (3,3);
  \draw[very thick] (0,\ymer) ellipse (0.35 and \bmer);
\end{scope}

% bottom vertical ellipse
\begin{scope}
  \clip (0,0) circle (\R);
  \clip (-3,-3) rectangle (0,3);
  \draw[dashed, very thick] (0,-\ymer) ellipse (0.35 and \bmer);
\end{scope}

\begin{scope}
  \clip (0,0) circle (\R);
  \clip (0,-3) rectangle (3,3);
  \draw[very thick] (0,-\ymer) ellipse (0.35 and \bmer);
\end{scope}

% horizontal ellipses
% horizontal ellipse parameters based on \R and \r
\pgfmathsetmacro{\ael}{0.5*(\R-\r)}     % horizontal radius = mid-radius
\pgfmathsetmacro{\bel}{0.25*(\R-\r)}    % vertical radius = 1/4 of thickness

% TOP (hidden) half of the horizontal right ellipse – dashed
\begin{scope}
  \clip (0,0) circle (\R);           % torus cross-section
  \clip (-3,0) rectangle (3,3);      % upper half-plane
  \draw[dashed] (\ymer,0) ellipse [x radius=\ael, y radius=\bel];
\end{scope}

% BOTTOM (visible) half – solid
\begin{scope}
  \clip (0,0) circle (\R);
  \clip (-3,-3) rectangle (3,0);     % lower half-plane
  \draw (\ymer,0) ellipse [x radius=\ael, y radius=\bel];
\end{scope}

% TOP (hidden) half of the horizontal left ellipse – dashed
\begin{scope}
  \clip (0,0) circle (\R);           % torus cross-section
  \clip (-3,0) rectangle (3,3);      
  \draw[dashed] (-\ymer,0) ellipse [x radius=\ael, y radius=\bel];
\end{scope}

% BOTTOM (visible) half – solid
\begin{scope}
  \clip (0,0) circle (\R);
  \clip (-3,-3) rectangle (3,0);     % lower half-plane
  \draw (-\ymer,0) ellipse [x radius=\ael, y radius=\bel];
\end{scope}

\coordinate (C) at (0,-2.9);

\draw[->]
  (C) ++(160:0.5 and 0.2)         
  arc[start angle=160,end angle=410,
      x radius=0.5,y radius=0.2];
\node[below] at (-0.9,3) {$\mathbb{T}^{2}$};
\node[right] at (0.6,-2.95) {$\mathbb Z_{2}$};

\draw[->,thick] (3,0) -- (4.5,0);

% pillowcase T^2 / Z_2 

\coordinate (P) at (5,  1);
\coordinate (Q) at (7,  1);
\coordinate (R) at (7, -1);
\coordinate (S) at (5, -1);

\draw[very thick] (P) .. controls (6,  0.75) .. (Q);   
\draw[very thick] (Q) .. controls (6.75, 0.0) .. (R);  
\draw[very thick] (R) .. controls (6, -0.75) .. (S);   
\draw[very thick] (S) .. controls (5.25, 0.0) .. (P);

\fill (P) circle (0.065);
\fill (Q) circle (0.065);
\fill (R) circle (0.065);
\fill (S) circle (0.065);

\node[above left]  at (P) {$a$};
\node[above right] at (Q) {$b$};
\node[below right] at (R) {$c$};
\node[below left]  at (S) {$d$};

\pgfmathsetmacro{\aH}{0.8}  
\pgfmathsetmacro{\bH}{0.25}   

\begin{scope}
  \clip (6,0) circle (1.3);          
  \clip (4.5,0) rectangle (7.5,3);   
  \draw[dashed] (6,0) ellipse [x radius=\aH, y radius=\bH];
\end{scope}
\begin{scope}
  \clip (6,0) circle (1.3);
  \clip (4.5,-3) rectangle (7.5,0);  
  \draw (6,0) ellipse [x radius=\aH, y radius=\bH];
\end{scope}

\node[below] at (6,-2.0) {$S^2(2,2,2,2)=\mathbb{T}^{2}/\mathbb Z_{2}$};

\end{tikzpicture}
    \caption{The pillowcase $S^2(2,2,2,2)=\mathbb{T}^2/\Z_2$ (right), where $\Z_2 = \Z/2\Z$ is generated by the $\pi$-rotation around the vertical axis of the torus $\mathbb{T}^2$ (left).}
    \label{fig:pillowcase}
\end{figure}
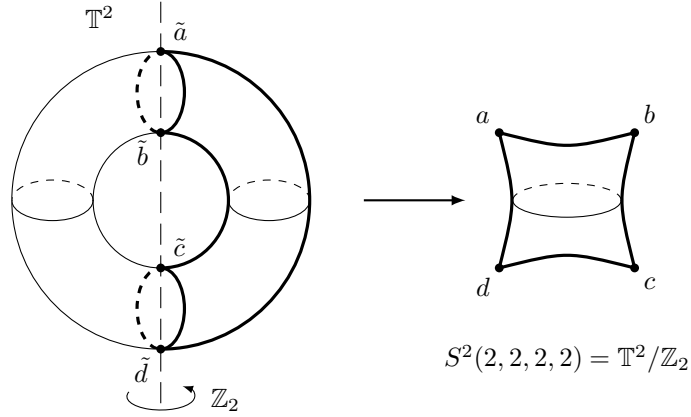

\subsection{Hyperbolic Dehn filling in the case of hyperbolic reflection groups}
Generally, in the case of non-orientable manifolds or orbifolds, one can simply pass to the orientable double cover and then discuss the Dehn filling process and geometric convergence in the oriented case.
However, for discrete groups $\Gamma$ generated by reflections in hyperplanes, the Dehn filling process deserves to be considered separately since it can be described more explicitly in terms of the corresponding Coxeter polyhedra. It turns out that the Dehn fillings $\Gamma_n$ obtained from $\Gamma$ can also be forced to be reflection groups, whose fundamental Coxeter polyhedra $P_n$ converge to the polyhedron $P$ by contracting a certain finite edge.

Let $Q \subset \mathbb{H}^3$ be a finite-volume Coxeter polyhedron with an edge $e$ with two finite vertices, and let the dihedral angle at $e$ be $\frac{\pi}{m}$ for some $m \ge 6$. Note that since $m\ge 6$, the dihedral angle at each edge adjacent to $e$ in $Q$ is $\frac{\pi}{2}$. By Andreev's theorem \cite{And70, And70b}, there always exists a finite-volume Coxeter polyhedron $Q'$ of the same combinatorial type and with the same dihedral angles as $Q$ except that the dihedral angle at $e$ in $Q'$ is replaced with $\frac{\pi}{n}$. If $n \geq m$, we say that $Q'$ is obtained from $Q$ via a {\em $\frac{\pi}{n}$-contraction} at $e$. Thus, we can denote the polyhedra with angle $\pi/n$ at $e$ by $P_n$. Andreev's theorem also allows to contract this edge $e$ to an ideal vertex $v_\infty$: combinatorially it means that we replace $e$ with a vertex of valence $4$ without changing the remaining combinatorics of the polyhedron. This contraction gives rise to a finite-volume Coxeter polyhedron $P$. We refer to Kolpakov~\cite{Kol12} for more details on contractions of edges.

What is important to us here is that we can view $P_n$ as having been obtained from $P$ via Dehn filling. Indeed, let $v$ be an ideal vertex of $P$, and let $v$ be the intersection of facets $F_1, F_2, F_3, F_4$ of $P$; see Figure~\ref{fig:Coxeter-Dehn}, the middle. These facets intersect any horosphere centered at $v$ in a Euclidean square, and therefore the corresponding reflections $r_1, r_2, r_3, r_4$ generate the parabolic reflection group $$\Gamma_v = \langle r_1, r_2, r_3, r_4 \mid r_i^2 = (r_1 r_2)^2 = (r_2 r_3)^2 = (r_3 r_4)^2 = (r_1 r_4)^2 = e\rangle,$$
which is isomorphic to $(\nicefrac{\Z}{2\Z} * \nicefrac{\Z}{2\Z}) \times (\nicefrac{\Z}{2\Z}  * \nicefrac{\Z}{2\Z} )$.

Note that the parabolic isometries $a = r_1 r_3$ and $b = r_2 r_4$ commute and generate a $\Z \times \Z$ subgroup of $\Gamma_v$. Algebraically, the $(p,q)$-Dehn filling in this case means the quotient by the normal closure of an isometry $a^p b^q$. Generally, it gives non-Coxeter groups because of the new relation $(r_1 r_3)^p (r_2 r_4)^q = e$. However, in the case of $(n,0)$- or $(0,n)$-Dehn fillings we obtain precisely the Coxeter groups $\Gamma_n$ or $\Gamma'_n$ generated by reflections in the walls of $P_n$ or $P'_n$, respectively, as in Figure \ref{fig:Coxeter-Dehn}. Indeed, this gives only an extra relation $(r_1 r_3)^n=e$ or $(r_2 r_4)^n = e$ in the quotient, therefore giving a Coxeter group. Denoting the generators in the quotient by $r_{i,n}$ or $r'_{i,n}$ in the case of $(n,0)$- and $(0,n)$-Dehn fillings, we obtain precisely the Coxeter groups $\Gamma_n$ and $\Gamma'_n$, generated by reflections in the walls of Coxeter polyhedra $P_n$ and $P'_n$. Combinatorially, these two sequences of polyhedra $P_n$ and $P'_n$ are obtained from the polyhedron $P$ by replacing a vertex $v$ with a compact edge in only two possible ways, as shown in Figure~\ref{fig:Coxeter-Dehn}.

As in the case of general Dehn fillings, such contractions of edges to infinity do not exist in higher dimensions (it was explicitly shown in \cite[Theorem 4.2]{BD23} for almost right-angled polyhedra).

\begin{figure}[ht]
    \centering
    \begin{tikzpicture}[scale=1.2]
\coordinate (1) at (0.5, 1) ;
  \coordinate (2) at (0.5, 0) ;
  \coordinate (3) at (1, 1.5) ;
  \coordinate (4) at (1, -0.5) ;
  \coordinate (a) at (0, -0.5) ;
  \coordinate (b) at (0, 1.5) ;
  \draw (-0.2,0.5) node {$r_{1,n}$};
  \draw (1.2,0.5) node {$r_{3,n}$};
  \draw (0.5,1.5) node {$r_{2,n}$};
  \draw (0.5,-0.5) node {$r_{4,n}$};
  \draw[line width=2.pt, red] (1) -- node [right] {$e$} (2) ;
  \draw[line width=1.pt] (1) -- (3) ; 
  \draw[line width=1.pt] (a) -- (2) ; 
  \draw[line width=1.pt] (b) -- (1) ;
  \draw[line width=1.pt] (4) -- (2) ; 
  
  \draw (0.5,-1.2) node {$P_n$};
  \draw (4.5-0.4,-1.2) node {$P$};
  \draw (9-0.5,-1.2) node {$P'_n$};
  \draw[
    thick,
    -{Stealth[length=4pt,width=4pt]},
    decorate,
    decoration={
      snake,
      amplitude=0.8mm,      % wave height
      segment length=3mm,   % wave length
      pre length=0pt,
      post length=1.2mm     % the arrowhead
    }
  ] (3-0.3, -1.2) -- (2-0.3,-1.2);
  \draw[
    thick,
    -{Stealth[length=4pt,width=4pt]},
    decorate,
    decoration={
      snake,
      amplitude=0.8mm,      % wave height
      segment length=3mm,   % wave length
      pre length=0pt,
      post length=1.2mm     % for the arrowhead
    }
  ] (6-0.5,-1.2) -- (7-0.5, -1.2);
  \fill[black] (1) circle (2pt) ;
  \fill[black] (2) circle (2pt) ; 
  \fill[black] (3) circle (2pt) ;
  \fill[black] (4) circle (2pt) ;  
  \fill[black] (a) circle (2pt) ;
  \fill[black] (b) circle (2pt) ;
  \draw[
    thick,
    -{Stealth[length=4pt,width=4pt]},
    decorate,
    decoration={
      snake,
      amplitude=0.8mm,      % wave height
      segment length=3mm,   % wave length
      pre length=0pt,
      post length=1.2mm     % the arrowhead
    }
  ] (3-0.3, 0.5) -- (2-0.3,0.5);
  \coordinate (6) at (4-0.4, 0);
  \coordinate (7) at (5-0.4, 0) ;
  \coordinate (8) at (4-0.4, 1) ;
  \coordinate (9) at (5-0.4, 1) ;
  \coordinate (5) at (4.5-0.4, 1/2);
  \draw (3.8-0.4,0.5) node {$r_1$};
  \draw (5.2-0.4,0.5) node {$r_3$};
  \draw (4.5-0.4,1) node {$r_2$};
  \draw (4.5-0.4,0) node {$r_4$};
  \draw[line width=1.pt] (6) -- (5) ;
  \draw[line width=1.pt] (5) -- (7) ;
  \draw[line width=1.pt] (8) -- (5) ; 
  \draw[line width=1.pt] (5) -- (9) ; 
  \fill[black] (6) circle (2pt);
  \fill[black] (7) circle (2pt) ; 
  \fill[black] (8) circle (2pt) ;
  \fill[black] (9) circle (2pt) ;
  \fill[red] (5) circle (2pt) node [right] {$v$}; 
  \draw[
    thick,
    -{Stealth[length=4pt,width=4pt]},
    decorate,
    decoration={
      snake,
      amplitude=0.8mm,      % wave height
      segment length=3mm,   % wave length
      pre length=0pt,
      post length=1.2mm     % the arrowhead
    }
  ] (6-0.5,0.5) -- (7-0.5,0.5);
  \coordinate (10) at (8-0.5, 0) ;
  \coordinate (11) at (10-0.5, 0) ;
  \coordinate (12) at (8-0.5, 1) ;
  \coordinate (13) at (10-0.5, 1) ;
  \coordinate (14) at (8.5-0.5, 1/2) ;
  \coordinate (15) at (9.5-0.5, 1/2) ;
  \draw (8-0.5-0.2,0.5) node {$r'_{1,n}$};
  \draw (10.2-0.5,0.5) node {$r'_{3,n}$};
  \draw (9-0.5,1.2) node {$r'_{2,n}$};
  \draw (9-0.5,-0.2) node {$r'_{4,n}$};
  \draw[line width=1.pt] (10) -- (14) ;
  \draw[line width=1.pt] (12) -- (14) ;
  \draw[line width=1.pt] (11) -- (15) ; 
  \draw[line width=1.pt] (15) -- (13) ; 
  \draw[line width=2.pt,red] (15) -- node [above] {$e'$} (14) ; 
  \fill[black] (10) circle (2pt) ;
  \fill[black] (11) circle (2pt) ; 
  \fill[black] (12) circle (2pt) ;
  \fill[black] (13) circle (2pt) ;
  \fill[red] (14) circle (2pt) ;
  \fill[red] (15) circle (2pt) ;
    \end{tikzpicture}
    \caption{Hyperbolic Dehn filling for Coxeter polyhedra: an ideal red vertex $v$ in the middle can be replaced with a compact edge in two ways, see the red edges $e$ (left) and $e'$ (right). Thus one obtains two sequences of polyhedra $P_n$ and $P'_n$ converging to $P$ as $n\to \infty$ along with contracting red edges $e,e'$ to the vertex $v$.}
    \label{fig:Coxeter-Dehn}
\end{figure}
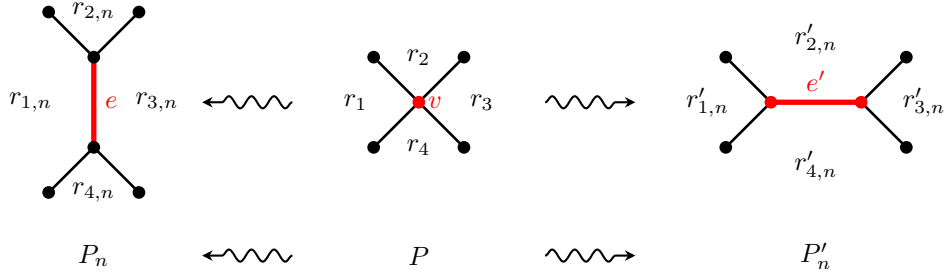

\subsection{Random walks}\label{sec:rw}

Let $\Gamma$ be a countable discrete group acting properly discontinuously on a metric space $(X, d)$ by isometries, and let $\mu$ be a probability measure on $\Gamma$. Consider a random walk
\[
\omega_n := g_1 \dots g_n,
\]
where the $(g_i)$ are i.i.d. with distribution $\mu$. We will say that a random walk $(\omega_n)$ has finite range if $\mu$ has finite support, and we will call it non-degenerate if the support of $\mu$ generates $\Gamma$ as a semigroup. 

Let $o \in X$ be a basepoint. Let $\mu^{*n} = \underbrace{\mu * \ldots * \mu}_{n \text{ times}}$ be the $n$th convolution of the measure $\mu$ with itself. Then one defines:

\begin{enumerate}
    \item the \textit{asymptotic entropy} $h(\mu)$, introduced by Avez \cite{Av72},
    \[
    h(\mu) := \lim_{n \to \infty} \frac{H(\mu^{*n})}{n} := \lim_{n \to \infty} \frac{1}{n} \sum_{g} -\mu^{*n}(g) \log \mu^{*n}(g),
    \]
    where $H(\cdot)$ denotes the {\em ordinary entropy} defined for any measure $\lambda$ as follows: for any countable set $A$ we have
    $$
H(\lambda) := - \sum_{a \in A} \lambda(a) \log \lambda(a);
    $$
    \item the \textit{drift} $\ell(\mu)$ of the random walk
    \[
    \ell(\mu) := \lim_{n \to \infty} \frac{d(o, \omega_n o)}{n}
    \]
    where, as long as $\mu$ has finite first moment, the limit exists almost surely and is independent of $\omega_n$ and $o$;
    \item the \textit{critical exponent} $v$ of the action
    \[
    v := \limsup_{R \to \infty} \frac{1}{R} \log \#\{ g \in \Gamma \ \mid\ d(o, g o) \leq R \}.
    \]
\end{enumerate}

These quantities are related by the \textit{fundamental inequality} due to Guivarc’h \cite{Gui80} (see also Vershik \cite{Ver01}, who suggested the name \textit{fundamental inequality})
\[
h(\mu) \leq \ell(\mu) v.
\]

Let us denote the support of $\mu$ by $S$. Recall that any random walk on a non-elementary subgroup of $\mathrm{Isom}(\mathbb{H}^d)$ will converge (see Maher--Tiozzo \cite{MT18}) to the geometric boundary $\mathbb{S}^{d-1} = \partial \mathbb{H}^d$. Equivalently, this means that for $\mathbb{P}$-a.e. sample path $(\omega_1, \omega_2, \dots)$ the limit $\lim_{n \rightarrow \infty} \omega_n o = \omega_\infty \in \mathbb{S}^{d-1}$ exists, where $\mathbb{P}$ is the probability measure which comes from $(S^\mathbb{N}, \mu^{ \otimes \mathbb{N}})$ via the pushforward 
$$
(g_1, g_2, g_3, \dots) \mapsto (g_1, g_1 g_2, g_1 g_2 g_3, \cdots).
$$
Thus, a random walk induces the (unique) {\em hitting measure} $\nu_\mu = i_*(\mathbb{P})$, where $i$ is the ($\mathbb{P}$-a.e. defined) mapping which takes a sample path $(\omega_n)_{n \in \mathbb{N}}$ to the limit $\omega_\infty \in S^2$. Another way to express this definition is to write

\begin{definition}
For any Borel set $A \subseteq \partial \mathbb{H}^d$,
$$
\nu_\mu(A) = \mathbb{P}\left(\left\{\lim_{n \to \infty} \omega_n o \in A\right\}\right).
$$
\end{definition}

We will be dealing with hyperbolic groups $\Gamma$ acting on $\mathbb{H}^2$ and $\mathbb{H}^3$, and in our setting,  we can make use of the following powerful criterion. We cite it here in the version for geometrically finite actions from \cite[Theorem 4.1]{GekhtmanTiozzo}, which extends the previous version \cite[Theorem 1.5]{blachere2011harmonic} for hyperbolic groups. In this context, a \emph{Patterson-Sullivan measure} for the action of $\Gamma$ on a $\delta$-hyperbolic space $X$ is a quasiconformal measure on $\partial X$ of dimension equal to the critical exponent $v$ of $\Gamma$.

\begin{theorem}\label{BHMTheorem 1.4}
Let $(X, d)$ be a $\delta$-hyperbolic, proper metric space, let $\Gamma$ be a geometrically finite group of isometries of $X$, and let $o\in X$ be a basepoint.

Let $\mu$ be a non-degenerate measure on $\Gamma$ with finite superexponential moment, let $\nu$ be its corresponding hitting measure, and let $\kappa$ be a Patterson-Sullivan measure on $\partial X$.

Then the following conditions are equivalent.
\begin{enumerate}
    \item The equality 
    $$h(\mu) = \ell(\mu) v$$ 
    holds.
    \item The measures $\nu$ and $\kappa$ are in the same measure class.
    \item The measures $\nu$ and $\kappa$ are in the same measure class with Radon--Nikodym derivatives bounded from above and below.
    \item There exists $C\geq 0$ such that for every $g \in \Gamma$, $$|d_G(e,g)-vd(o,go)|\leq C,$$
    where $d_G$ denotes the Green metric for the random walk driven by $\mu$.
\end{enumerate}
\end{theorem}

\begin{corollary}[\cite{GekhtmanTiozzo}, Corollary 4.2] \label{singularcorollary}
If $\Gamma\curvearrowright X$ is not convex cocompact then $\kappa$ and $\nu$ are mutually singular.
\end{corollary}

In this paper, we are interested in $X = \mathbb{H}^d$ and $\Gamma < \textup{Isom}(\mathbb{H}^d)$ a lattice.
In this case it is well known that the Lebesgue measure $\lambda$ on $\partial \mathbb{H}^d$ is a Patterson-Sullivan measure. 

To see why, recall that if $g \in \textup{Isom}(\mathbb{H}^d)$ then (see e.g. \cite[Proposition 3.9]{Quint})
$$\frac{d g \lambda}{d \lambda}(\xi) = e^{-(d-1) \beta_\xi(go, o)} \qquad \forall \xi \in \partial \mathbb{H}^d$$
so $\lambda$ is quasiconformal and, since $\Gamma$ is a lattice, $v = n-1$. 

The corollary then immediately shows that 
if $\Gamma < \textup{Isom}(\mathbb{H}^d)$ is a nonuniform lattice and $\mu$ an admissible, finitely supported measure on $\Gamma$, then the hitting measure $\nu$ is singular with respect to Lebesgue, and moreover
$h(\mu) < (d-1)\ell(\mu).$

\section{Auxiliary results}

\subsection{Inequality for the entropy}

\begin{proposition}\label{prop:entropy-inequality}
    Let $\Gamma$ be a discrete group, and $N \triangleleft \Gamma$ its normal subgroup. Let $\mu_\Gamma$ be a probability measure on $\Gamma$ and $\mu_\Lambda$ be the pushforward of $\mu_\Gamma$ to $\Lambda = \Gamma/N$ through the homomorphism $\varphi : \Gamma \to \Gamma/N$. Then $h(\mu_{\Lambda}) \leq h(\mu_\Gamma)$.
\end{proposition}
\begin{proof}
First, we note that $\mu_\Lambda(\gamma N) = \sum_{h \in \gamma N} \mu_\Gamma(h)$. Using the monotonicity of $\log$, we have
$-\log \mu_\Lambda(\gamma N) = -\log\left(\sum_{h \in \gamma N} \mu_\Gamma(h)\right) \leq -\log \mu_\Gamma(h)$ and therefore
\begin{equation*}
\begin{split}
    -\mu_\Lambda(\gamma N) \log \mu_\Lambda(\gamma N) = &\ -\left(\sum_{h \in \gamma N} \mu_\Gamma(h)\right) \cdot \log\left(\sum_{h \in \gamma N} \mu_\Gamma(h)\right) \\
    = &\ -\sum_{h \in \gamma N} \mu_\Gamma(h)  \log \mu_\Lambda(\gamma N) \\ 
    \leq &\ -\sum_{h \in \gamma N} \mu_\Gamma(h) \log \mu_\Gamma(h).
\end{split}
\end{equation*}

Hence,
\begin{equation*}
\begin{split}
    H(\mu_\Lambda) =&\ -\sum_{\lambda \in \Lambda = \Gamma/N} \mu_\Lambda(\lambda) \log \mu_\Lambda(\lambda) = - \sum_{\gamma N} \mu_\Lambda(\gamma N) \log \mu_\Lambda(\gamma N) \\
    \leq &\ -\sum_{\gamma N} \sum_{h \in \gamma N} \mu_\Gamma(h) \cdot \log \mu_\Gamma(h) \\
    = &- \sum_{g \in \Gamma} \mu_\Gamma(g) \cdot \log\left( \mu_\Gamma(g) \right) = H(\mu_\Gamma).
\end{split}
\end{equation*}

A very similar argument shows that $H(\mu_\Lambda^{*n}) \leq H(\mu_\Gamma^{*n})$. Dividing the latter inequality by $n$ and taking the limit $\lim_{n \to \infty}$, we complete the proof.
\end{proof}

\subsection{Algebraic vs geometric convergence of hyperbolic lattices}\label{sec:alg-geom-conv}

Here we summarize several specific facts about algebraic and geometric convergence that we will need to prove our main theorems. For some of the claims, we did not find proofs in the literature, so we provide a detailed argument here, although all these facts are known. In particular, part (1) is stated by Kapovich in his book, see \cite[Exercise 8.12]{Kap-book}. Under the assumption that all $\Gamma_n$ are lattices, part (2) follows from \cite[Prop. E.2.3]{BP-lectures} and \cite[Prop. E.3.1]{BP-lectures}, and part (3) follows from \cite[Theorem E.4.8]{BP-lectures}, as we explain in the proof below.

\begin{theorem}\label{th:geom-alg-conv}
    Let $\Gamma_n, \Gamma < \Isom(\HH^d)$, $d \ge 3$, be torsion-free discrete groups. Assume that $\Gamma_n \to \Gamma$ geometrically and $\Gamma$ is a lattice. Then the following assertions hold.
    \begin{enumerate}
        \item There exists $N > 0$ such that for all $n > N$ the manifolds $M_n=\HH^d/\Gamma_n$ have finite volume, i.e. $\Gamma_n$ are also lattices.
        \item If, moreover, $\Gamma$ is a uniform lattice, or non-uniform but $d \ge 4$, then $\Gamma_n$ are eventually isomorphic to $\Gamma$, i.e. there is $N > 0$ such that $\Gamma_n \cong \Gamma$ for all $n > N$.
        \item If $d=3$ and $\Gamma$ is non-uniform, then there is $N > 0$ such that for $n>N$ each $\Gamma_n$ is obtained from $\Gamma$ by hyperbolic Dehn filling and therefore $\Gamma_n = \rho_n(\Gamma)$ which algebraically converge to the initial representation $\rho(\Gamma)=\Gamma$.
    \end{enumerate}
\end{theorem}
\begin{proof}
    All parts follow from the thick-thin decomposition and the tools of geometric convergence, but part (1) also requires some standard facts from algebraic topology. By assumption $\Gamma$ is a lattice, and thus the thick part $M^{\ge \varepsilon}$ of $M$, where $\varepsilon \le \varepsilon_d$, is compact by Proposition \ref{prop:thick-thin-finite-volume}.
    \medskip
 
    We will first consider the case where $\Gamma$ is a uniform lattice, i.e. where the manifold $M = \HH^d/\Gamma$ is closed. Taking the compact set $K = M$, we obtain by Corollary \ref{cor:phi-n} that there is $N > 0$ such that for all $n > N$ there exist the bi-Lipschitz diffeomorphisms $\phi_n : M \hookrightarrow M_n = \HH^d/\Gamma_n$. Then $\phi_n(M) = N_n$ is a closed $d$-submanifold of the connected $d$-manifold $M_n$. This implies that $N_n$ is both an open and closed subset of the connected manifold $M_n$, whence  $N_n = M_n$. Thus, $M_n$ itself is a compact hyperbolic $d$-manifold diffeomorphic to $M$, and then the Mostow rigidity implies that $\Gamma_n$ is isomorphic to $\Gamma$ for all $n > N$. This proves part (1) and (2) for the case of $\Gamma$ being a uniform lattice. 
    \medskip

    Henceforth, assume that $M=\HH^d/\Gamma$ has finite volume and is noncompact, i.e. it has (finitely many) cusps. In this case, we can choose $\varepsilon \le \varepsilon_d$ so that the thick part $M^{\ge \varepsilon}$ is a compact connected manifold whose boundary components are closed flat $(d-1)$-manifolds $F_i$ (see Remark \ref{rem:thin-small} based on Proposition \ref{prop:thick-thin-finite-volume}). We will separately treat the cases $d=3$ and $d>3$. For simplicity, we can pass simultaneously to orientation preserving subgroups of $\Gamma_n$ and $\Gamma$; so without loss of generality we assume that all manifolds $M_n$ and $M$ are orientable. 

    Consider the compact set $K = M^{\ge \varepsilon}$ and the submanifolds $N_n=\phi_n(K)$ for the bi-Lipschitz diffeomorphisms $\phi_n \colon K \hookrightarrow M_n$ (which exist for all $n>N$ for some $N>0$) from Corollary \ref{cor:phi-n}. The complement $M_n - N_n$ has connected components $C_i$ each having the boundary flat components $F'_i$, diffeomorphic to closed flats $F_i$; see Figure \ref{fig:Mn-to-M}. One can show that each $C_i$ is an aspherical manifold based on the fact that $M_n$, $N_n$, and $F'_i$ are all aspherical.

    \medskip
    {\em Case $d=3$.}  In this dimension, the situation is simpler: $F_i$ are just flat $2$-dimensional tori $T_i$. We will denote their images by $T'_i$. Since $T'_i$ is the boundary of $C_i$, by \cite[Proposition 9.3.6]{Mar} we have only two options. 

    (a) $T'_i$ bounds a solid torus, so $C_i$ itself is a solid torus and therefore compact.

    (b) $T'_i$ is incompressible in $C_i$, and hence $\pi_1$-injective by Hempel \cite[Lemma~6.1]{Hem-book}. Since $T_i$ is $\pi_1$-injectuve for $M$, we thus obtain that $T'_i$ is $\pi_1$-injective for the whole $M_n$. Then $T'_i$ is a cusp cross-section, since the subgroup $\Z \times \Z = \pi_1 (T'_i)$  of $\pi_1 (M_n)=\Gamma_n$ must preserve a point on the ideal boundary by \cite[Corollary 4.2.4]{Mar}. Hence, the entire $C_i$ is just a finite-volume truncated cusp of $M_n$.

    Thus, $M_n$ consists of the compact block $N_n = \phi_n(M^{\geq \varepsilon})$ and of several pieces $C_i$ each of them being either compact or finite volume. Then $M_n$ has finite volume. 

    \begin{figure}
        \centering
        \begin{tikzpicture}[scale=0.9]

  % Blue geodesics:
  \def\rBlue{1.2426406871}   % = 3 * tan(pi/8)
  % Centers of blue geodesics
  \coordinate (CB1) at ( 2.2961005942,  2.2961005942);
  \coordinate (CB2) at (-2.2961005942,  2.2961005942);
  \coordinate (CB3) at (-2.2961005942, -2.2961005942);
  \coordinate (CB4) at ( 2.2961005942, -2.2961005942);

  % Red common perpendiculars:
  \def\rRed{2.5226892458}    % = 3 * 0.8408964...
  % Centers of red geodesics:
  \coordinate (CRright) at ( 3.9196888946,  0.0);   % between B4 and B1
  \coordinate (CRtop)   at ( 0.0,  3.9196888946);   % between B1 and B2
  \coordinate (CRleft)  at (-3.9196888946,  0.0);   % between B2 and B3
  \coordinate (CRbot)   at ( 0.0, -3.9196888946);   % between B3 and B4

  \begin{scope}
    \clip (0,0) circle (3);
    \draw[very thick,blue] (CB1) circle (\rBlue);
    \draw[very thick,blue] (CB2) circle (\rBlue);
    \draw[very thick,blue] (CB3) circle (\rBlue);
    \draw[very thick,blue] (CB4) circle (\rBlue);
  \end{scope}

  % red perpendiculars
  % Between CB1 and CB2 (top)
  \begin{scope}
    \clip (0,0) circle (3);
    \draw[very thick,red]
      (CRtop) ++(-118.5:\rRed)
      arc[start angle=-118.5, end angle=-61.5, radius=\rRed];
  \end{scope}

  % Between CB2 and CB3 (left)
  \begin{scope}
    \clip (0,0) circle (3);
    \draw[very thick,red]
      (CRleft) ++(-28.5:\rRed)
      arc[start angle=-28.5, end angle=28.5, radius=\rRed];
  \end{scope}

  % Between CB3 and CB4 (bottom)
  \begin{scope}
    \clip (0,0) circle (3);
    \draw[very thick,red]
      (CRbot) ++(61.5:\rRed)
      arc[start angle=61.5, end angle=118.5, radius=\rRed];
  \end{scope}

  % Between CB4 and CB1 (right)
  \begin{scope}
    \clip (0,0) circle (3);
    \draw[very thick,red]
      (CRright) ++(151.5:\rRed)
      arc[start angle=151.5, end angle=208.5, radius=\rRed];
  \end{scope}

  \coordinate (A) at ( 7, 3);
  \coordinate (B) at ( 7, 3);
  \coordinate (C) at (4, 0);
  \coordinate (D) at ( 7,-3);

  \draw[very thick,blue]
    (A) arc[start angle=180, end angle=270, radius=3];

  \draw[very thick,blue]
    (B) arc[start angle=0, end angle=-90, radius=3];

  \draw[very thick,blue]
    (C) arc[start angle=90, end angle=0, radius=3];

  \draw[very thick,blue]
    (D) arc[start angle=180, end angle=90, radius=3];

  %cusp 1
  \begin{scope}
    \draw[very thick,red]
      (4,0) ++(-15:1.5)
      arc[start angle=-15, end angle=15, radius=1.5];
  \end{scope}

  %cusp 2  
  \begin{scope}
    \draw[very thick,red]
      (7,3) ++(-104.5:1.5)
      arc[start angle=-104.5, end angle=-75.5, radius=1.5];
  \end{scope}

  %cusp 4
  \begin{scope}
    \draw[very thick,red]
      (7,-3) ++(75:1.5)
      arc[start angle=75, end angle=105, radius=1.5];
  \end{scope}

  % cusp 3
  \begin{scope}
    \draw[very thick,red]
      (10,0) ++(165:1.5)
      arc[start angle=165.5, end angle=195, radius=1.5];
  \end{scope}

\draw[->, thick] (2,-3.6) -- node[above] {$n \to \infty$} (5,-3.6);

\draw [thick, left hook-latex](5,-2) -- node[above] {$\phi_n$} (3,-2);

\node at (0,-3.6) { $M_n$}  ;

\node at (0,-2.3) { $C_4$}  ;
\node at (-2.3,0) { $C_1$}  ;
\node at (2.3,0) { $C_3$}  ;
\node at (0,2.3) { $C_2$}  ;

\node at (0,0) { $N_n$}  ;

\node at (7,-3.6) { $M$}  ;
\node at (7,0) { $M^{\ge \varepsilon}$}  ;

\node at (0,-1.1) { {\color{red} $F'_4$}}  ;
\node at (0,1.1) { {\color{red} $F'_2$}}  ;
\node at (1.1,0) { {\color{red} $F'_3$}}  ;
\node at (-1.1,0) { {\color{red} $F'_1$}}  ;

\node at (7,-1.1) { {\color{red} $F_4$}}  ;
\node at (7,1.1) { {\color{red} $F_2$}}  ;
\node at (8.1,0) { {\color{red} $F_3$}}  ;
\node at (7-1.1,0) { {\color{red} $F_1$}}  ;

\end{tikzpicture}
        \caption{Geometric convergence of hyperbolic $d$-manifolds $M_n$ (left) to a finite-volume manifold $M$ with cusps (right). The thick part $M^{\ge \varepsilon}$ of $M$ is a connected compact manifold with closed flat boundary components $F_i$. The image of $M^{\ge \varepsilon}$ under the diffeomorphism $\phi_n$ gives $N_n \subset M_n$ (left) bounded by $F'_i$, and $M_n - N_n$ is a disjoint union of $C_i$. }
        \label{fig:Mn-to-M}
    \end{figure}
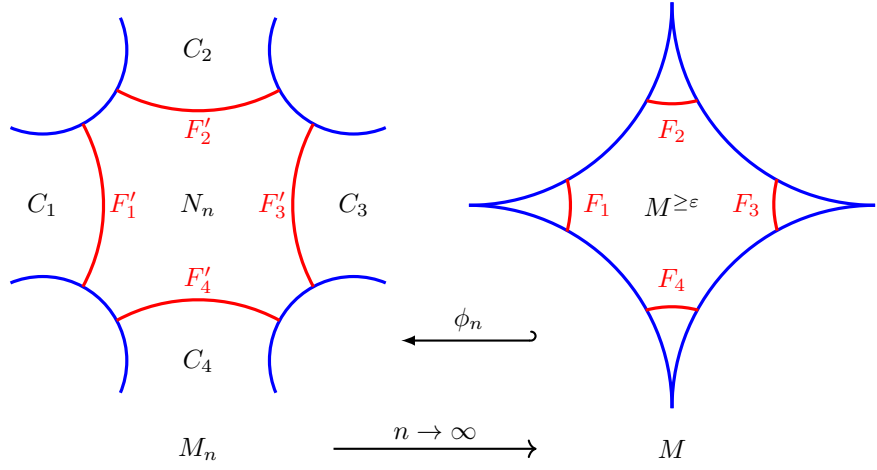
    
    \medskip
    {\em Case $d>3$.} In this case, the flat boundary components of $M^{\ge \varepsilon}$ are some closed $(d-1)$-manifolds $F_i$ that are finitely covered by flat tori $T_i$ of rank $d-1$. These flat manifolds $F_i$ are embedded $\pi_1$-injectively in $M$, and, as in the case $d=3$, we again analyze their images $F'_i = \phi_n(F_i)$ in $M_n$. 

    We again have two options.

    (a) $F'_i$ is not $\pi_1$-injective in $C_i$. Then $C_i$ is compact as follows from the following two lemmas.

    \vspace{-1.5em}
    \begin{adjustwidth}{0.7cm}{0.7cm}
    \begin{lemma}\label{lem:pi1-homology}
    Let $X$ be a connected aspherical $d$-manifold and $F \subset X$ a closed, connected, embedded flat hypersurface. Assume that the inclusion $i \colon F \hookrightarrow X$ is not $\pi_1$-injective. Then $i_*([F])$ is a torsion element of $H_{d-1}(X;\mathbb Z)$. In particular, the image of $[F]$ in $H_{d-1}(X;\mathbb Q)$ is zero.
    \end{lemma}

    \begin{proof}
    We write $\Gamma=\pi_1(X)$. Since $X$ is aspherical, it is a $K(\Gamma,1)$-space and $\Gamma$ is torsion-free. By the Bieberbach theorem, $F$ is finitely covered by a torus $\mathbb T^{d-1}$.

    \medskip
    \noindent\underline{Case 1}: $F$ is a torus $\mathbb T^{d-1}$.
    Then
    $i_*:\pi_1(\mathbb T^{d-1})\cong\mathbb Z^{d-1}\to \Gamma$ has nontrivial kernel by assumption, so its image is a free abelian group $\mathbb Z^{k}$ with $k<d-1$.

    Since $\mathbb T^{d-1}$ and $X$ are Eilenberg--MacLane spaces $K(\mathbb Z^{d-1},1)$ and $K(\Gamma,1)$, the map $i$ is, up to homotopy, determined by $i_*$ and factors through the classifying space $K(\mathbb Z^{k},1)\cong\mathbb T^k$. Thus there exist maps
    $$
    \mathbb T^{d-1} \xrightarrow{g} \mathbb T^{k} \xrightarrow{f} X
    $$
    inducing the homomorphisms $\mathbb Z^{d-1} \twoheadrightarrow \mathbb Z^{k} \to \Gamma$ on $\pi_1$, such that $i$ is homotopic to $f\circ g$.

    On homology with $\mathbb Z$-coefficients we have $H_{d-1}(\mathbb T^{d-1};\mathbb Z)\cong\mathbb Z$, generated by the fundamental class $[\mathbb T^{d-1}]$, while $H_{d-1}(\mathbb T^{k};\mathbb Z) = 0$ since $k<d-1$. Hence
    $g_*[\mathbb T^{d-1}] = 0 \in H_{d-1}(\mathbb T^k;\mathbb Z),$
    and therefore
    $$
    i_*[\mathbb T^{d-1}] = f_* g_*[\mathbb T^{d-1}] = 0 \in H_{d-1}(X;\mathbb Z).
    $$

\medskip
\noindent\underline{Case 2}: $F$ is only virtually a torus. In this case, $F$ is a closed flat $(d-1)$-manifold with a finite covering
$p:\mathbb T^{d-1}\longrightarrow F.$
The covering $p$ induces a short exact sequence
$$
1 \longrightarrow \pi_1(\mathbb T^{d-1}) \cong \mathbb Z^{d-1}
\longrightarrow \pi_1(F) \longrightarrow A \longrightarrow 1,
$$
where $A$ is finite and $\pi_1(F)$ is a Bieberbach group, hence torsion-free.

Let $\varphi = i_*:\pi_1(F)\to\pi_1(X)$. By assumption $\varphi$ is not injective, so
$$
K := \ker(\varphi) \subset \pi_1(F)
$$
is nontrivial. Suppose first that $K\cap \pi_1(\mathbb T^{d-1})$ is trivial. Then the restriction of the projection $\pi_1(F)\to A$ to $K$ is injective, so $K$ embeds into the finite group $A$ and is therefore finite. But $\pi_1(F)$ is torsion-free, so it cannot contain a nontrivial finite subgroup. This contradiction shows that
$K\cap \pi_1(\mathbb T^{d-1})$ is non-trivial.

Thus, the composition
$$
\pi_1(\mathbb T^{d-1}) \xrightarrow{p_*} \pi_1(F) \xrightarrow{i_*} \pi_1(X)
$$
has nontrivial kernel, i.e.\ the induced map
$i\circ p : \mathbb T^{d-1} \to X$ is not $\pi_1$-injective. By Case~1, the induced map on $(d-1)$-dimensional homology satisfies
$$
(i\circ p)_*[\mathbb T^{d-1}] = 0 \in H_{d-1}(X;\mathbb Z).
$$

On the other hand, since $p$ is a finite covering of degree $m>0$, we have
$$
p_*[\mathbb T^{d-1}] = m\,[F] \in H_{d-1}(F;\mathbb Z).
$$
Therefore
$$
0 = (i\circ p)_*[\mathbb T^{d-1}]
= i_* p_*[\mathbb T^{d-1}]
= i_*\bigl(m\,[F]\bigr)
= m\,i_*[F] \in H_{d-1}(X;\mathbb Z).
$$
Hence $i_*[F]$ is a torsion element of $H_{d-1}(X;\mathbb Z)$. In particular, its image in $H_{d-1}(X;\mathbb Q)$ is zero.
\end{proof}
\end{adjustwidth}

    \vspace{-1.5em}
    \begin{adjustwidth}{0.7cm}{0.7cm}
    \begin{lemma}\label{lem:homology-compact}
    If $F = \partial X$ is the compact connected boundary of a connected $d$-manifold $X$ and $[F]=0$ in the $H_{d-1}(X)$, then $X$ is compact.
    \end{lemma}
    \begin{proof}
    We can work over the field $\Z_2=\Z/2\Z$ or over the field $\Q$ if $X$ is oriented.
    Consider the double $DX$ of $X$ along the boundary. Applying the Mayer--Vietoris to the classical decomposition $DX = A \cup B$, where $A \cap B = F \times (-1,1) \simeq F$, and $A$ and $B$ are thickenings of $X$ by the collars $F \times (-1,0]$ and $F \times [0,1)$, respectively, we obtain that $[F]\in \mathrm{im}(\delta)$ where $\delta : H_d(DX) \to H_{d-1}(F).$ Indeed, since $F$ has dimension $d-1$, $H_d(F)=0$, so the relevant part of the Mayer--Vietoris sequence is
    $$
    0 \to H_d(A)\oplus H_d(B) \longrightarrow H_d(DX) \longrightarrow H_{d-1}(F)
    \xrightarrow{j_*} H_{d-1}(A)\oplus H_{d-1}(B).
    $$
    Here $j_*$ is induced by the inclusions $F\hookrightarrow A$ and
    $F\hookrightarrow B$. Identifying $A$ and $B$ with $X$, we can write
    $$
    j_* = (i_*, -i_*): H_{d-1}(F)\longrightarrow H_{d-1}(X)\oplus H_{d-1}(X),
    $$
    where $i_*:H_{d-1}(F)\to H_{d-1}(X)$ is the map induced by the inclusion $i: F\hookrightarrow X$. By assumption, $i_*[F]=0$ in $H_{d-1}(X)$, so
    $$
    j_*[F] = (i_*[F], -i_*[F]) = (0,0).
    $$
    Thus $[F]\in \ker(j_*)$. By exactness, $\ker(j_*)$ is the image of $\delta$, whence $[F]\in \mathrm{im}(\delta)$.  On the other hand, $[F] \ne 0$ in $H_{d-1}(F)$ and thus $H_d(DX) \ne 0$.

    Now $DX$ is a connected $d$-manifold without boundary and has nonzero
    top-dimensional homology $H_d(DX)$. It is a standard fact that a connected noncompact $d$-manifold without boundary has $H_d=0$ (see, for example, Hatcher \cite[Proposition~3.29]{Hat}). Therefore $DX$ must be compact.
    \end{proof}    
    \end{adjustwidth}

    (b) $F'_i$ is $\pi_1$-injective in $C_i$. Then it is also $\pi_1$-injective for the entire $M_n$ and we conclude as in the case (b) for $d=3$ by showing the component $C_i$ is a finite-volume truncated cusp of $M_n$.
    
    This proves part (1).

    \medskip
    Part (2) follows from \cite[Prop. E.2.3]{BP-lectures} and \cite[Prop. E.3.1]{BP-lectures}.

    Part (3) follows from \cite[Theorem E.4.8]{BP-lectures} which states the following. Recall that $\mathcal{F}$ is the set of all complete finite-volume hyperbolic $3$-manifolds with the topology of geometric convergence. Let $\mathscr{F}_3(c) \subset \mathcal{F}$ be the subset of manifolds $M$ with $\vol M \le c$. Then there exists a finite collection of finite-volume manifolds $X_1, \ldots, X_k$ with cusps such that $X \in \mathscr{F}_3(c)$ is obtained from some $X_i$ by hyperbolic Dehn surgery.
    
    In our case, we have $M_n = \HH^3/\Gamma_n$ geometrically converge to $M = \HH^3/\Gamma$, whence, for $n$ large enough, the volumes $\vol M_n$ are bounded from above and thus we can apply \cite[Theorem E.4.8]{BP-lectures}.
\end{proof}

\subsection{Compatible measures on convergent sequences of lattices}\label{sec:compatible}

\begin{definition}
    Assume that $\Gamma_n = \rho_n(\Gamma)$ algebraically converge to $\Gamma=\rho_\infty(\Gamma)$. Measures $\mu_n$ on $\Gamma_n$ are said to be compatible if, for all $n$, we have $\mu_n = (\rho_n)_* \mu$ of some measure $\mu$ on $\Gamma$.
\end{definition}

\begin{lemma}\label{lem:compatible-measures}
    Let $\{\Gamma_n\}_{n=1}^\infty$ be a sequence of pairwise non-conjugate lattices in $\Isom(\HH^3)$ geometrically converging to a lattice $\Gamma$. Then there is $N>0$ such that for any measure $\mu$ on $\Gamma$, the remaining tail $\{\Gamma_n\}^\infty_{n=N+1}$ of this sequence can be equipped with compatible measures $\mu_n$. Moreover, if for some $m > N$, the lattice $\Gamma_m$ is equipped with a measure $\mu_m$ compatible with some measure $\mu$ on $\Gamma$, then we can equip all the other members of the subsequence $\{\Gamma_n\}^\infty_{n=N+1}$ with measures $\mu_n$ compatible with $\mu_m$.
\end{lemma}
\begin{proof}
    First of all, note that Theorem \ref{th:geom-alg-conv} shows that there is $N > 0$ such that for $n>N$ each $\Gamma_n$ is obtained from $\Gamma$ by hyperbolic Dehn surgery and therefore $\Gamma_n = \rho_n(\Gamma)$ converge algebraically to the representation $\rho_\infty(\Gamma)=\Gamma$. Then we set $\mu_n := (\rho_n)_* \mu$, which proves the first part. For the second part, we note that $\mu = (\rho_m)^* \mu_m$ on $\Gamma$, and then the first part applies.
\end{proof}

\section{Proof of main theorems }

\subsection{Kleinian groups and proof of Theorem~\ref{th:main-Kleinian}}

Let $\Gamma_n < \Isom(\HH^3)$ be an infinite sequence of lattices, geometrically converging to a lattice $\Gamma$. The assumption of $\Gamma_n$ being pairwise non-conjugate implies, by Theorem \ref{th:geom-alg-conv}, that $\Gamma$ is non-uniform and all but finitely many $\Gamma_n$ are obtained from $\Gamma$ by hyperbolic Dehn surgery. Hence, removing finitely many exceptions, we may assume that $\Gamma_n = \rho_n(\Gamma)$ algebraically converge to $\Gamma=\rho_\infty(\Gamma)$. Lemma~\ref{lem:compatible-measures} shows that we can equip our subsequence by compatible measures, and this can be done for an arbitrary given measure $\mu_m$ on any $\Gamma_m$.

From now on, we assume that the entire sequence $\Gamma_n$ is equipped with compatible measures $\mu_n$ obtained via pushforward from some measure $\mu$ on $\Gamma$. Let us denote by $h_n$ the entropy $h(\mu_n)$, and set $\ell_n = \ell(\mu_n)$, $\ell = \ell(\mu)$. The hitting measures driven by $\mu_n$ and $\mu$ will be denoted by $\nu_n$ and $\nu$, respectively.

Using Proposition~\ref{prop:entropy-inequality} and the fact that $\Gamma_n = \rho_n(\Gamma)=\Gamma/\ker\rho_n$, we obtain that $h_n \le h$. From Theorem \ref{BHMTheorem 1.4} we know that $\nu_n \perp \lambda$ if and only if $h_n < \ell_n v_n$. Here $v_n = v = 2$ for all $n$. The drift is continuous for random walks on subgroups of $\mathrm{GL}_2(\CC)$ (that is, the continuity of the drift applies to finitely subgroups of both $\mathrm{PSL}_2(\CC)$ and $\mathrm{PSL}_2(\R)$), as shown by the recent result of Bocker-Neto and Viana \cite{BNV17}, whence we have that $\ell_n \to \ell$ as $n \to \infty$. Finally, Theorem \ref{BHMTheorem 1.4} (see also Corollary \ref{singularcorollary} and the discussion after it) implies that $h<2\ell$, since the lattice $\Gamma < \Isom(\HH^3)$ is non-uniform.
Therefore,
$$
2\ell_n \xrightarrow{n\to \infty} 2\ell > h \ge h_n.
$$
Hence, there is $N>0$ such that $2\ell_n > h_n$ for all $n>N$, 
so $\nu_{\mu_n}$ is singular for these values of $n$. \qed

\subsection{Hyperbolic reflection groups and proof of Theorem \ref{th:main-Coxeter}}
In this case, the proof is a simple (special) version of the proof Theorem \ref{th:main-Kleinian}. Indeed, if the edge $e$ in Coxeter polyhedra $P_n$ is formed by facets corresponding to reflections $r_1^{(n)}$ and  $r_2^{(n)}$, our reflection groups satisfy: $\Gamma_n = \Gamma/N_n$, where $N_n$ is the normal closure of $(r_1 r_2)^n \in \Gamma$.
By Proposition \ref{prop:entropy-inequality}, we have the inequality $h_n < h$ for all $n$. The conclusion of the proof remains the same as in Theorem \ref{th:main-Kleinian}. \qed

\subsection{Fuchsian groups and proof of Theorem \ref{th:main-Fuchsian}}

For Fuchsian groups, the result is stated already in the setting of algebraic convergence. This allows to apply Proposition \ref{prop:entropy-inequality} directly and then again conclude in the same way as it was done in the proof of Theorem \ref{th:main-Kleinian}. \qed

\section{Examples}\label{sec:examples}

Throughout this section, $\Sigma_{g,a,b}$ denotes an orientable surface of genus $g \ge 1$ with $a$ punctures and $b$ boundary components; if $b=0$, we will simply write $\Sigma_{g,a}$, and $\Sigma_g$ denotes a closed genus-$g$ surface.

\subsection{Geometric convergence of infinite-covolume Fuchsian groups to a lattice}

Let $\Sigma$ denote an infinite-area surface of genus $1$ with one {\em funnel}, i.e. infinite end, and let $\alpha$ be a simple closed geodesic separating the funnel, see Figure \ref{fig:infinite-to-S110}. Cutting off the funnel along $\alpha$ would give us a surface $\Sigma_{1,0,1}$ in the above notation.

To construct a geometrically converging sequence of infinite-covolume Fuchsian groups, we take a surface $\Sigma_n$ with the same topology as $\Sigma$, but replace $\alpha$ with a closed curve $\alpha_n$ such that the length of $\alpha_n$ tends to $0$ as $n \to \infty$. For sufficiently large $n$ and an appropriate choice of $\varepsilon > 0$, the thin part $(\Sigma_n)^{\leq \varepsilon}$ of $\Sigma_n$ is the neighborhood of $\alpha_n$ (bounded by blue arcs, see the top of Figure~\ref{fig:infinite-to-S110}). Then the thick part $(\Sigma_n)^{\geq \varepsilon}$ has two connected components, one of them being the surface $\Sigma_{1,0,1}$. Denote this component by $T_n$. Taking the basepoints $x_n \in T_n$ so that they ``stay apart'' (uniformly) from the geodesic $\alpha_n$, we obtain the pointed Gromov--Hausdorff convergence of $(\Sigma_n, x_n)$ to $(\Sigma_{1,1}, x)$ for some $x \in {\Sigma_{1,1}}$. By Theorem \ref{th:equivalent-convergences}, this is equivalent to the geometric convergence of $\pi_1 (\Sigma_n)$ to $\pi_1 (\Sigma_{1,1})$, which acts as a lattice on $\HH^2$. We note that the groups $\pi_1 (\Sigma_n)$ and $\pi_1 (\Sigma_{1,1})$ are all abstractly isomorphic to $\mathrm{F}_2$, a free group of rank $2$.

\begin{figure}
    \centering
    \includegraphics[width=0.5\linewidth]{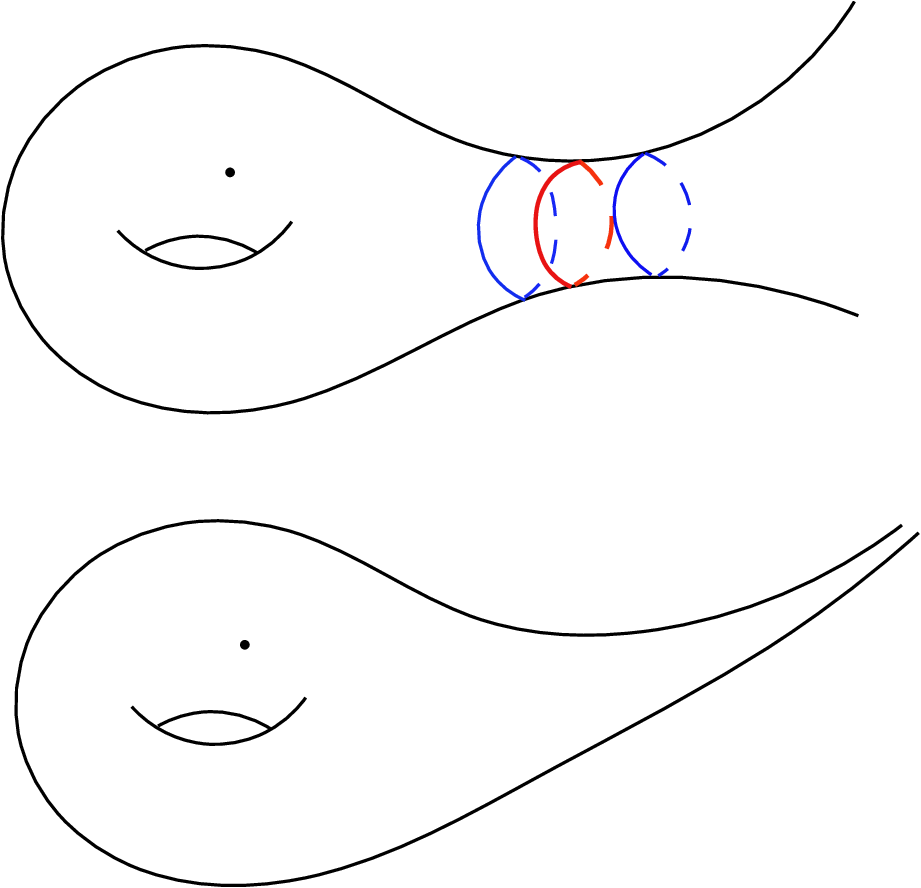}
    \caption{Pinching the red geodesic $\alpha$ on $\Sigma_{1,0,1}$ (top), we obtain the geometric limit being the once-punctured torus $\Sigma_{1,1}$ (bottom).}
    \label{fig:infinite-to-S110}
\end{figure}

This geometric convergence can also be justified in the setting of the pointed Gromov--Hausdorff convergence of the corresponding Dirichlet fundamental polyhedra. Consider the Poincare disc model of $\HH^2$ with $o$ the origin. We choose a metric on $\Sigma_{1,1}$ such that the Dirichlet fundamental domain $D_{\Gamma}(o)$ of $\Gamma=\pi_1(\Sigma_{1,1})$ is the regular ideal quadrilateral shown in Figure \ref{fig:Dirichlet} (right). Then we take another quadrilateral bounded by four pairwise disjoint geodesics such that all the common perpendiculars $\alpha_n$ between the adjacent (consecutive) sides have the same length, see Figure \ref{fig:Dirichlet} (left). These red perpendiculars $\alpha_n$ project under the $\Gamma_n = \pi_1 (\Sigma_n)$-action precisely to the red geodesic $\alpha_n$ illustrated in Figure \ref{fig:infinite-to-S110} (top). Shrinking simultaneously all these $\alpha_n$, we obtain the pointed Gromov--Hausdorff convergence $D_{\Gamma_n}(o) \xrightarrow{pGH} D_{\Gamma}(o)$. By Theorem \ref{th:equivalent-convergences}, this is equivalent to the geometric convergence of $\Gamma_n$ to $\Gamma$.
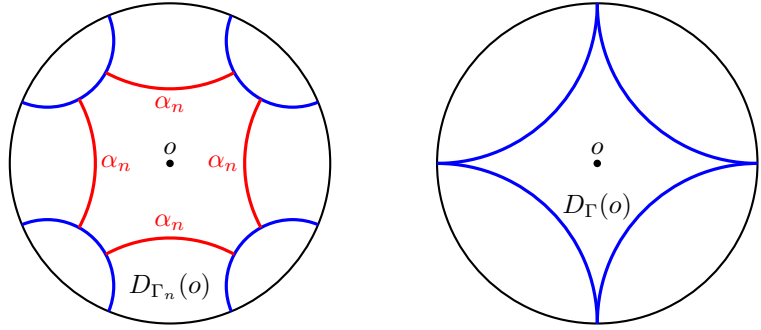
\begin{figure}
    \centering
    \begin{tikzpicture}[scale=0.71]
  \draw[thick] (8,0) circle (3);
  \coordinate (A) at ( 8, 3);
  \coordinate (B) at ( 8, 3);
  \coordinate (C) at (5, 0);
  \coordinate (D) at ( 8,-3);

  \draw[very thick,blue]
    (A) arc[start angle=180, end angle=270, radius=3];

  \draw[very thick,blue]
    (B) arc[start angle=0, end angle=-90, radius=3];

  \draw[very thick,blue]
    (C) arc[start angle=90, end angle=0, radius=3];

  \draw[very thick,blue]
    (D) arc[start angle=180, end angle=90, radius=3];

  \draw[thick] (0,0) circle (3);

  % Blue geodesics:
  \def\rBlue{1.2426406871}   % = 3 * tan(pi/8)
  % Centers of blue geodesics 
  \coordinate (CB1) at ( 2.2961005942,  2.2961005942);
  \coordinate (CB2) at (-2.2961005942,  2.2961005942);
  \coordinate (CB3) at (-2.2961005942, -2.2961005942);
  \coordinate (CB4) at ( 2.2961005942, -2.2961005942);

  % Red common perpendiculars:
  \def\rRed{2.5226892458}    % = 3 * 0.8408964...
  % Centers of red geodesics:
  \coordinate (CRright) at ( 3.9196888946,  0.0);   % between B4 and B1
  \coordinate (CRtop)   at ( 0.0,  3.9196888946);   % between B1 and B2
  \coordinate (CRleft)  at (-3.9196888946,  0.0);   % between B2 and B3
  \coordinate (CRbot)   at ( 0.0, -3.9196888946);   % between B3 and B4

  \begin{scope}
    \clip (0,0) circle (3);
    \draw[very thick,blue] (CB1) circle (\rBlue);
    \draw[very thick,blue] (CB2) circle (\rBlue);
    \draw[very thick,blue] (CB3) circle (\rBlue);
    \draw[very thick,blue] (CB4) circle (\rBlue);
  \end{scope}

  % red perpendiculars

  % Between CB1 and CB2 (top)
  \begin{scope}
    \clip (0,0) circle (3);
    \draw[very thick,red]
      (CRtop) ++(-118.5:\rRed)
      arc[start angle=-118.5, end angle=-61.5, radius=\rRed];
  \end{scope}

  % Between CB2 and CB3 (left)
  \begin{scope}
    \clip (0,0) circle (3);
    \draw[very thick,red]
      (CRleft) ++(-28.5:\rRed)
      arc[start angle=-28.5, end angle=28.5, radius=\rRed];
  \end{scope}

  % Between CB3 and CB4 (bottom)
  \begin{scope}
    \clip (0,0) circle (3);
    \draw[very thick,red]
      (CRbot) ++(61.5:\rRed)
      arc[start angle=61.5, end angle=118.5, radius=\rRed];
  \end{scope}

  % Between CB4 and CB1 (right)
  \begin{scope}
    \clip (0,0) circle (3);
    \draw[very thick,red]
      (CRright) ++(151.5:\rRed)
      arc[start angle=151.5, end angle=208.5, radius=\rRed];
  \end{scope}
\node at (0,-2.3) { $D_{\Gamma_n}(o)$}  ;
\node at (8,-0.8) {$D_{\Gamma}(o)$ }  ;

\node at (0,-1.1) { {\color{red} $\alpha_n$}}  ;
\node at (0,1.1) { {\color{red} $\alpha_n$}}  ;
\node at (1,0) { {\color{red} $\alpha_n$}}  ;
\node at (-1,0) { {\color{red} $\alpha_n$}}  ;

\fill (0,0) circle (0.07) node[above] {$o$};
\fill (8,0) circle (0.07) node[above] {$o$};
\end{tikzpicture}
    \caption{When the length of all red arcs $\alpha_n$ (left) tends simultaneously to $0$, the domains $D_{\Gamma_n}(o)$ (left) shrink to the regular ideal quadrilateral $D_{\Gamma}(o)$ (right), and therefore $D_{\Gamma_n}(o) \xrightarrow{pGH} D_{\Gamma}(o)$.}
    \label{fig:Dirichlet}
\end{figure}

\subsection{Non-equivalence of algebraic and geometric convergence of cocompact Fuchsian groups}

Consider an orientable closed surface $\Sigma_g$ with a separating closed geodesic $\alpha$ on it such that it divides $\Sigma_g$ into the union of $\Sigma_{1,0,1}$ and $\Sigma_{g-1, 0, 1}$, see Figure \ref{fig:s3-to-s11} illustrating the case $g=3$. Pinching $\alpha$ gives rise to an infinite sequence of closed surfaces $(\Sigma_g)_n$ that geometrically converge to a once-punctured torus $\Sigma_{1,1}$. We can justify it similarly to the previous example. Indeed, denote the corresponding geodesic on $(\Sigma_g)_n$ by $\alpha_n$, where $\mathrm{Len}(\alpha_n) \to 0$. For sufficiently large $n$ and an appropriate choice of $\varepsilon > 0$, the thick parts $(\Sigma_g)_n^{\ge \varepsilon}$ have two  connected components, where one of them, denoted by $T_n$ corresponds to the left piece $(\Sigma_{1,0,1})_n$. Taking the basepoints $x_n \in Th_n$ so that they stay apart from the geodesic $\alpha_n$, we obtain the pointed Gromov--Hausdorff convergence of $((\Sigma_g)_n, x_n)$ to $(\Sigma_{1,1}, x)$ for some $x \in {\Sigma_{1,1}}$. By Theorem \ref{th:equivalent-convergences}, this is equivalent to the geometric convergence $\pi_1 ((\Sigma_g)_n) \to \pi_1 (\Sigma_{1,1}) \cong \mathrm{F}_2$, a free group of rank $2$.

\begin{figure}
    \centering
    \includegraphics[width=0.75\linewidth]{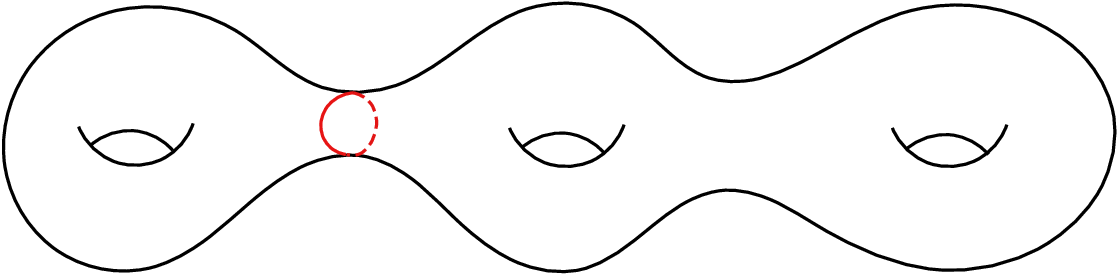}
    \caption{Pinching the red arc on $\Sigma_3$ gives a sequence $(\Sigma_3)_n$ that converges geometrically to the once-punctured torus~$\Sigma_{1,1}$.}
    \label{fig:s3-to-s11}
\end{figure}

On the other hand, all $\pi_1 ((\Sigma_g)_n)$ are isomoprhic to $\pi_1 (\Sigma_g)$, being a one-relator group of four generators, and thus, there are no surjections from $\pi_1 (\Sigma_{1,1}) \cong \mathrm{F}_2$, a free group of rank $2$, onto $\pi_1 (\Sigma_g)$.

\section{Concluding remarks and further questions}

Our main results (Theorems \ref{th:main-Kleinian}, \ref{th:main-3-orbifold}, and \ref{th:main-Fuchsian}) heavily rely on {\em algebraic convergence}, which is essential in the proof of Proposition \ref{prop:entropy-inequality}. In dimension $d = 3$, this convergence is equivalent to the geometric one, as follows from Theorem \ref{th:geom-alg-conv}, part (3). On the other hand, in Section \ref{sec:examples}, we provide an example showing that there are geometrically converging cocompact Fuchsian groups that do not algebraically converge to the same limit. In Theorem \ref{th:geom-alg-conv}, part (1), we also showed that infinite-covolume lattices in $\Isom(\HH^d)$, $d \ge 3$, cannot geommetrically converge to a lattice, while it can happen if $d=2$ (see an example in Section \ref{sec:examples}). These phenomena motivate the following questions which we state below.

Let $\Gamma_n$ be a sequence of finitely generated Fuchsian groups that geometrically converge to a lattice $\Gamma < \mathrm{PSL}_2(\R)$. Assume that the measures $\mu_n$ are supported on generating sets $S_n$ (of $\Gamma_n$) that converge geometrically to a generating set $S$ of $\Gamma$ such that for each sequence of generators $s_n \in S_n$ converging to a generator $s \in S$, the corresponding numbers $p_n$ on $s_n$ also converge to some number $p$ on $s$.

\begin{question}\label{q:conv-entropy}
    Is it true that $h_n \to h$? Is it true at least when all but finitely many $\Gamma_n$ are lattices?
\end{question}

\noindent A positive answer to the first question aims to replace Proposition \ref{prop:entropy-inequality} in the proof of the singularity conjecture for groups that are geometrically close to a group for which the singularity conjecture is proven.

\begin{question}\label{q:sing-Fuchs}
    Assume that all but finitely many $\Gamma_n$ are infinite-covolume. Does it imply that the hitting measures $\nu_{\mu_n}$ are singular with respect to the Patterson--Sullivan measure $\kappa$ on $\partial \HH^2=\mathbb{S}^1$?
\end{question}

\noindent Assuming a positive answer to Question~\ref{q:conv-entropy}, we can answer Question \ref{q:sing-Fuchs} in affirmative using the fact that $h_n \to h$, $\ell_n \to \ell$, and $v_n \to v$ (the convergence of critical exponents follows from \cite{Sul84} and \cite{Tuk84} as explained in \cite[Corollary 4.7 (b)]{KK23}).

\bibliography{biblio.bib}{}
\bibliographystyle{siam}

\end{document}